\documentclass{amsart}
\usepackage[T1]{fontenc}
\usepackage{graphicx}
\usepackage{amsmath,amssymb}
\usepackage{empheq}
\usepackage{amsthm}
\usepackage{paralist}
\usepackage{hyperref}
\usepackage{caption}
\hypersetup{hidelinks}

\newtheorem{thm}{Theorem}[section]

\newtheorem{prop}[thm]{Proposition}

\theoremstyle{definition}

\theoremstyle{remark}
\newtheorem{rem}[thm]{Remark}
\numberwithin{equation}{section}

\begin{document}
\title[]{Tumor boundary instability induced by pressure feedback}%

\author[Y. Feng]{Yu Feng}%
\address{Yu Feng:  Department of Mathematics, Great Bay University, Dongguan, China, 523000, P.R.China }
\email{fengyu@gbu.edu.cn}

\author[Y. Zhang]{Yeyu Zhang}%
\address{Yeyu Zhang: School of Mathematics, Shanghai University of Finance and Economics, Shanghai 200433, China}
\email{zhangyeyu@mail.shufe.edu.cn}

\author[K. Zhou]{Kewei Zhou}%
\address{Kewei Zhou: School of Mathematics, Shanghai University of Finance and Economics, Shanghai 200433, China}
\email{2024213176@stu.sufe.edu.cn}

\date{\today}


\thanks{}

\begin{abstract}
We consider a Hele--Shaw-type tumor growth model obtained from the
incompressible limit of a porous-medium equation. The limiting model
preserves the patch structure of the tumor density and therefore admits
a single free-boundary formulation. We investigate how a
pressure-feedback parameter associated with the homeostatic pressure
affects boundary stability. Using asymptotic analysis, we study boundary
perturbations of finite-radius tumors and planar traveling fronts and
derive the corresponding stability criteria. For planar
fronts, the stability thresholds are further shown to be bifurcation
points from which nonsymmetric traveling waves emerge. Our results indicate
that pressure feedback can induce boundary instability in parameter
regimes where the model without pressure feedback remains stable.
\end{abstract}

\maketitle
\section{Introduction}

The morphology and evolution of the tumor boundary reflect the combined
effects of cell proliferation, nutrient availability, and mechanical
interactions, and provide a macroscopic description of tumor development.
Understanding when a smooth tumor boundary remains stable or develops
irregular structures is therefore relevant to the mechanisms governing
tumor growth and invasion. Since these processes involve coupled biological
and mechanical effects acting across multiple spatial and temporal scales,
mathematical modeling, model analysis,
and numerical simulation have become important tools for investigating tumor
boundary dynamics.

Mathematical models of tumor growth have been developed within several
continuum frameworks, including reaction--diffusion systems
\cite{roose2007mathematical}, phase-field models
\cite{chatelain2011emergence,lowengrub2009nonlinear}, and free-boundary
models \cite{friedman1999analysis}. Reaction--diffusion models
describe the evolution of tumor-cell density and its interaction with
nutrients and other biological components, whereas free-boundary models
represent the tumor directly as an evolving domain; phase-field models
provide an alternative diffuse-interface description. A natural connection
between density-based models and sharp-interface models is provided by the
incompressible limit of porous-medium-type equations. As the pressure law
becomes increasingly stiff, the cell density approaches a saturated patch,
and the boundary of its support evolves according to a Hele--Shaw-type
free-boundary problem. Since the seminal work of Perthame, Quir\'os, and
V\'azquez \cite{perthame2014hele}, this limiting process has become an active
direction in the mathematical study of tumor growth. A more detailed overview
of the related developments is provided in the following paragraph.

At the analytical level, subsequent studies established the incompressible
limit in different solution frameworks and under increasingly general
assumptions. The limiting Hele--Shaw problem has been characterized using
weak solutions \cite{mellet2017hele} and viscosity solutions
\cite{kim2016free}. The convergence has also been extended to general initial
densities \cite{kim2018porous} and growth laws without a monotonicity
assumption \cite{guillen2022hele}. More recently, Tong and
Zhang \cite{tong2025convergence} proved that the free boundaries themselves
converge in the Hausdorff distance and obtained geometric estimates for the
limiting interface. The limiting procedure has further been
justified for models incorporating viscosity
\cite{perthame2015incompressible}, nutrient coupling \cite{david2021free},
and convective transport \cite{david2024incompressible}. Free-boundary
limits have also been studied for reaction--cross-diffusion
systems \cite{dou2022tumor}. For porous-medium equations with chemotactic
aggregation, the limit was studied first for the Patlak--Keller--Segel
model and its stationary states
\cite{he2023pks}, and subsequently in the presence of pressure-dependent
growth \cite{he2025chemotaxis}. Models allowing necrotic regions lead to
an obstacle problem for the pressure, whose coincidence set represents
the necrotic core \cite{dou2024tumor}.

In parallel, numerical methods have been developed to
capture the free-boundary dynamics arising in the stiff-pressure regime.
These include methods connecting cell-density models with their free-boundary
limits \cite{liu2018analysis}, accurate front-capturing schemes
\cite{liu2018accurate}, and numerical approaches for resolving boundary
propagation speeds \cite{liu2021toward}. These developments have recently
been extended to a stabilized computational framework for tumors with
necrotic cores \cite{feng2026tumor}, stochastic asymptotic-preserving schemes
\cite{jiang2025efficient}, and multi-fidelity methods for uncertainty
quantification \cite{yu2026multifidelity}. Related inverse problems have
also received increasing attention. Bayesian inversion has been used to infer
model parameters under different pressure laws \cite{feng2024unified}, while
Lipschitz stability with respect to the porous-medium exponent has been
established to support Bayesian parameter inference
\cite{debiec2025lipschitz}. Data-driven methods based on physics-informed
neural networks and DeepONet have further been applied to parameter
identification from experimental data \cite{liu2027data}.

Another line of research concerns boundary stability in
free-boundary tumor models.
A pioneering contribution was made by Cristini, Lowengrub, and Nie
\cite{cristini2003nonlinear}, who combined linear analysis of boundary
perturbations with boundary-integral simulations to investigate their
nonlinear evolution, including fingering instabilities and topological
changes.
For the
classical free-boundary models originating with Greenspan
\cite{greenspan1972models}, Friedman and his
collaborators
established the asymptotic stability of radially symmetric solutions
\cite{friedman1999analysis,friedman2006asymptotic} and investigated
symmetry-breaking bifurcations leading to nonsymmetric stationary states
\cite{friedman2001symmetry,friedman2006bifurcation,
friedman2008stability}. Subsequent studies have examined a
range of mechanisms affecting boundary stability and symmetry breaking,
including delayed cell proliferation \cite{zhao2020impact}, chemotaxis
\cite{lu2022bifurcation}, heterogeneous nutrient supply
\cite{lu2020complex}, and growth inhibitors
\cite{wu2012bifurcation}. Many of these models impose the Laplace--Young condition
\(p=\sigma\kappa\), where \(\sigma\) is the surface-tension coefficient and
\(\kappa\) is the mean curvature, or related boundary conditions involving
curvature. Accordingly, the coefficient associated with the curvature term
often serves as the stability or bifurcation parameter.

For tumor models arising from porous-medium equations and
their incompressible limits, Kim and
collaborators studied the regularity and stability of the resulting tumor
patches. In particular, regularizing dynamics and boundary regularity were
established in the absence of nutrient diffusion and cell death
\cite{jacobs2023tumor}, while, for small nutrient diffusion, the patch
boundary was shown to remain uniformly close in the Hausdorff distance to
the smooth boundary obtained in the zero-diffusion limit
\cite{kim2023tumor}. More
recently, the space-time geometry of Hele--Shaw flows with a nonnegative
source was studied in
\cite{collins2026spacetime}, where the regular and singular behavior of the
free boundary, including collision and hole-closing events, was characterized;
these results apply in particular to nutrient-coupled tumor growth models.
In our previous works, we adapted
perturbative boundary-stability analysis
\cite{feng2023tumor} and bifurcation methods \cite{feng2025nonsymmetric}
developed for classical free-boundary models to PME-derived Hele--Shaw models.
A three-dimensional extension of the finite-radius
boundary-instability analysis was subsequently developed in
\cite{liu2024three}.
A fundamental difference is that the pressure in the
PME-derived models considered in these works vanishes
on the tumor boundary. Hence, the
curvature coefficient employed in the classical models is absent and cannot
serve as a parameter governing the transition to boundary instability. It is
therefore necessary to identify other biological or physical mechanisms
capable of destabilizing the boundary. Our previous results showed that
nutrient consumption, together with the nutrient-supply regime, provides such
a mechanism.

In addition to nutrient consumption, pressure-dependent
inhibition of cell proliferation may also affect boundary stability.
Cell proliferation increases local pressure,
whereas elevated pressure suppresses cell division and may arrest net growth
at the homeostatic pressure. Ye and Lin \cite{ye2024fingering} showed that this
pressure--growth coupling can induce fingering instability in a proliferating
cell collective: boundary protrusions relax the local pressure and thereby
enhance cell growth near their tips. Motivated by this observation, we incorporate a
pressure-dependent inhibitory term into a nutrient-coupled porous-medium tumor
growth model and investigate how the corresponding pressure-feedback
parameter affects the boundary stability of its incompressible Hele--Shaw
limit. In particular, the limit as pressure feedback vanishes is not uniform
with respect to the tumor radius: even an arbitrarily small positive
pressure-feedback parameter can destabilize sufficiently large tumors in
regimes where the model without pressure feedback remains stable.

The remainder of this paper is organized as follows. In Section~2, we formally
derive the Hele--Shaw model considered in this paper from a porous-medium
equation. In Section~3, we investigate the boundary stability of finite-radius
tumors and planar traveling fronts through asymptotic analysis. In Section~4,
we establish the existence of nonsymmetric traveling waves by local
bifurcation analysis. Finally, Section~5 summarizes the main results and
discusses directions for future research.

\subsection*{Acknowledgments}

The authors thank Qingyou He for helpful discussions.

Y.~Feng was partially supported by the National Key R\&D Program of China,
Project No. 2021YFA1001200, and the NSFC Youth Program, Grant No. 12501669.

Y.~Zhang was supported in part by the National Natural
Science Foundation of China (NSFC), Grant Nos. 12401562, 12671494, and
12241103, and by the Ministry of Science and Technology of the People's
Republic of China, Grant No. 2024YFA1012401.

All mathematical calculations and proofs in this work
were carried out by the authors. Generative artificial intelligence
tools were used solely for language editing and for assisting in
checking the calculations and proofs. All AI-assisted suggestions were
independently reviewed and verified by the authors, who take full
responsibility for the content of this work.

\section{Model introduction}

\subsection{A porous-medium model for tumor growth}

We consider a tumor growing in $\mathbb{R}^2$. For $m>1$, let
$\rho_m=\rho_m(t,x)$ denote the tumor-cell density and let
\begin{equation}
\label{eqn:D_set}
D_m(t):=\{x\in\mathbb{R}^2:\rho_m(t,x)>0\}
\end{equation}
denote the region occupied by tumor cells. The pressure
$p_m=p_m(t,x)$ generated by cell
compression is related to the density through the power law
\begin{equation}
\label{eqn:pressure_law}
p_m(\rho_m)=\rho_m^m.
\end{equation}
The cell velocity is assumed to obey Darcy's law,
$u_m=-\nabla p_m$. Consequently, conservation of cell mass gives the
reactive porous-medium equation
\begin{equation}
\label{eqn:PME_density}
\partial_t\rho_m-\nabla\cdot(\rho_m\nabla p_m)
=\rho_m\Phi(p_m,c_m),
\qquad x\in\mathbb{R}^2,\qquad t>0,
\end{equation}
where $c_m=c_m(t,x)$ is the nutrient concentration and
$\Phi$ is the net cell proliferation rate. We take the growth function
to be
\begin{equation}
\label{eqn:growth_rate_fcn}
\Phi(p,c)=\frac{c}{c_B}-\frac{p}{p_M},
\end{equation}
where $c_B>0$ is the background nutrient concentration and $p_M>0$
is the homeostatic pressure scale. Biologically, the
nutrient-dependent term promotes cell proliferation, whereas the
pressure-dependent term represents mechanical inhibition of growth.
The strict decrease of $\Phi$ with respect to $p$ satisfies the
structural assumption in \cite{perthame2014hele}. This condition is
essential for applying the incompressible-limit and patch-preservation
results recalled in the next subsection.

Equation \eqref{eqn:PME_density} is degenerate parabolic. In particular,
compactly supported initial data propagate with finite speed,
so that $D_m(t)$ remains bounded on every finite time interval.
At a regular point $x\in\partial D_m(t)$, the boundary
expansion speed is given by Darcy's law as
\begin{equation}
\label{eqn:finite_m_boundary_velocity}
V_n(t,x)=u_m(t,x)\cdot n(t,x)
=-\nabla p_m(t,x)\cdot n(t,x).
\end{equation}
$n(t,x)$ is the outward unit normal vector at $x$.

We couple \eqref{eqn:PME_density} to the \emph{in vivo} nutrient regime
considered in \cite{feng2023tumor}. The saturated region associated
with the porous-medium equation is defined by
\begin{equation}
\label{eqn:S_set}
S_m(t):=\{x\in\mathbb{R}^2:\rho_m(t,x)\geq 1\}.
\end{equation}
Within $S_m(t)$, we assume nutrient is consumed at a rate
proportional to the cell density. In its complement, nutrient is supplied by the surrounding
tissue at a rate proportional to the nutrient deficit $c_B-c_m$ and to
the available volume fraction $(1-\rho_m)_+$. The nutrient concentration
therefore satisfies
\begin{subequations}
\label{eqn:finite_m_invivo}
\begin{alignat}{2}
\tau\partial_t c_m-\Delta c_m+\lambda\rho_m c_m
    &=0,
    \quad &&\text{in } S_m(t),
    \label{eqn:finite_m_c_in}\\
\tau\partial_t c_m-\Delta c_m
    &=(1-\rho_m)_+(c_B-c_m),
    \quad &&\text{in } \mathbb{R}^2\setminus S_m(t),
    \label{eqn:finite_m_c_out}
\end{alignat}
\end{subequations}
where $\lambda>0$ is the nutrient consumption rate and $\tau>0$ is the
ratio of the nutrient-diffusion time scale to the cell-evolution time
scale. Here and below, \(s_+:=\max\{s,0\}\) denotes the
positive part of \(s\).
The concentration and its normal derivative are continuous
across $\partial S_m(t)$.

\subsection{The incompressible limit and patch solutions}

We next pass formally to the incompressible, or mesa, limit
$m\to\infty$. This limit connects the density model introduced in the
preceding subsection to a Hele--Shaw free boundary problem. Indeed, multiplying
\eqref{eqn:PME_density} by $m\rho_m^{m-1}$ and using
\eqref{eqn:pressure_law}, we obtain
\begin{equation}
\label{eqn:finite_m_pressure}
\partial_t p_m
=|\nabla p_m|^2
+mp_m\bigl(\Delta p_m+\Phi(p_m,c_m)\bigr).
\end{equation}
Let $(\rho_\infty,p_\infty,c_\infty)$ denote the limit
of $(\rho_m,p_m,c_m)$ as $m\to\infty$. The leading-order balance in
\eqref{eqn:finite_m_pressure} as $m\to\infty$ yields the complementarity
relation
\begin{equation}
\label{eqn:complementarity}
p_\infty\bigl(\Delta p_\infty
    +\Phi(p_\infty,{c_\infty})\bigr)=0.
\end{equation}
The limiting density satisfies
\begin{equation}
\label{eqn:rho_infty}
\partial_t\rho_\infty
-\nabla\cdot(\rho_\infty\nabla p_\infty)
=\rho_\infty\Phi(p_\infty,{c_\infty})
\end{equation}
in the sense of distributions. Moreover, $\rho_\infty$ and $p_\infty$
satisfy the Hele--Shaw graph relation
\begin{equation}
\label{eqn:HS_graph}
p_\infty\in P_\infty(\rho_\infty),
\qquad
P_\infty(s)=
\begin{cases}
\{0\}, & 0\leq s<1,\\[2pt]
[0,\infty), & s=1.
\end{cases}
\end{equation}
In particular, $0\leq\rho_\infty\leq1$, and $p_\infty$ can
be positive only in the saturated region $\Omega(t)$ associated with the
Hele--Shaw problem, defined by
\begin{equation}
\label{eqn:saturated_region}
\Omega(t):=\{x\in\mathbb{R}^2:\rho_\infty(t,x)=1\}.
\end{equation}

For notational simplicity, from now on we set
$\gamma:=1/p_M>0$. Within the saturated region $\Omega(t)$,
\eqref{eqn:growth_rate_fcn}, \eqref{eqn:rho_infty}, and
\eqref{eqn:HS_graph}, together with Darcy's law, give the pressure
problem
\begin{subequations}
\label{eqn:limitpressure}
\begin{alignat}{2}
-\Delta p_\infty+\gamma p_\infty
    &=\frac{c_\infty}{c_B},
    \quad &&\text{in } \Omega(t),
    \label{eqn:pressure_equation}\\
p_\infty
    &=0,
    \quad &&\text{on } \partial\Omega(t),
    \label{bc:p_zero}\\
V_n
    &=-\nabla p_\infty\cdot n,
    \quad &&\text{on } \partial\Omega(t),
    \label{eqn:boundary_speed}
\end{alignat}
\end{subequations}
where $n$ is the outward unit normal to $\Omega(t)$ and $V_n$ is the
normal velocity of its boundary.

As long as the nutrient concentration remains positive,
the strong maximum principle applied to
\eqref{eqn:pressure_equation} yields $p_\infty>0$ throughout
$\Omega(t)$. Since \eqref{eqn:rho_infty} preserves the zero-density
phase, this positivity and the Hele--Shaw graph relation
\eqref{eqn:HS_graph} imply that an initial density patch
retains its patch structure throughout the incompressible-limit
evolution. We may therefore write
\begin{equation}
\label{eqn:patch_solution}
\rho_\infty(t,x)=\chi_{\Omega(t)}(x),
\end{equation}
where $\chi_A$ denotes the characteristic function of a
set $A$. The evolution of the patch boundary is governed by
\eqref{eqn:boundary_speed}.

We next consider the nutrient evolution. Since nutrient
diffusion occurs on a faster time scale than tumor growth, we impose
the quasi-static approximation by setting $\tau=0$ in
\eqref{eqn:finite_m_invivo}. As established in
\cite{feng2025nonsymmetric}, this quasi-static reduction preserves the
patch structure \eqref{eqn:patch_solution}. For clarity, we decompose
$c_\infty$ with respect to the sharp interface $\partial\Omega(t)$. Let
$c_\infty^{\mathrm{(i)}}:=c_\infty|_{\Omega(t)}$ and
$c_\infty^{\mathrm{(o)}}:=
c_\infty|_{\mathbb{R}^2\setminus\overline{\Omega(t)}}$ denote the
interior and exterior nutrient concentrations, respectively. For a
patch solution, the right-hand side of \eqref{eqn:pressure_equation}
can therefore be written more specifically as
$c_\infty^{\mathrm{(i)}}/c_B$, and the nutrient concentrations satisfy
\begin{subequations}
\label{eqn:invivonutrient}
\begin{alignat}{2}
-\Delta {c_\infty^{\mathrm{(i)}}}
    +\lambda {c_\infty^{\mathrm{(i)}}}
    &=0,
    \quad &&\text{in } \Omega(t),
    \label{eqn:c_in}\\
-\Delta {c_\infty^{\mathrm{(o)}}}
    +{c_\infty^{\mathrm{(o)}}}
    &=c_B,
    \quad &&\text{in } \mathbb{R}^2\setminus\overline{\Omega(t)},
    \label{eqn:c_out}\\
{c_\infty^{\mathrm{(i)}}}
    &={c_\infty^{\mathrm{(o)}}},
    \quad &&\text{on } \partial\Omega(t),
    \label{bc:c_continuity}\\
\partial_n {c_\infty^{\mathrm{(i)}}}
    &=\partial_n {c_\infty^{\mathrm{(o)}}},
    \quad &&\text{on } \partial\Omega(t).
    \label{bc:flux_continuity}
\end{alignat}
\end{subequations}
Note that \eqref{eqn:c_out} yields
${c_\infty^{\mathrm{(o)}}}(t,x)\to c_B$ as $|x|\to\infty$.
Moreover, at each fixed time, the maximum principle for the transmission problem
ensures that $c_\infty^{\mathrm{(i)}}$ and
$c_\infty^{\mathrm{(o)}}$ are positive in their respective domains.

Although the derivation above has been presented formally, the passage
from \eqref{eqn:PME_density}--\eqref{eqn:finite_m_invivo} to the coupled
free boundary system \eqref{eqn:limitpressure}--\eqref{eqn:invivonutrient},
including the incompressible limit and the quasi-static nutrient
approximation, can be justified rigorously within the frameworks of
\cite{perthame2014hele,david2021free,david2024incompressible,
feng2025nonsymmetric}. Since this justification follows the existing arguments and is not the
purpose of the present work, we omit its proof.
In the subsequent sections, we focus on the boundary stability of
\eqref{eqn:limitpressure}--\eqref{eqn:invivonutrient}; for notational
simplicity, the subscript $\infty$ on the pressure and nutrient variables
is henceforth omitted.

\section{Asymptotic analysis of boundary stability}

In \cite{feng2023tumor}, the boundary stability of the Hele--Shaw model
\eqref{eqn:limitpressure}--\eqref{eqn:invivonutrient} was studied in the
special case \(p_M=+\infty\), or equivalently \(\gamma=0\). It was shown
that sufficiently strong nutrient consumption may destabilize
low-frequency boundary modes, while high-frequency modes remain stable.
In the present work, we fix the nutrient consumption rate \(\lambda>0\)
and investigate how the pressure-feedback parameter
\(\gamma>0\) affects the boundary stability. We consider both
finite-radius solutions and planar traveling waves.

\subsection{Finite-radius scenario}
\label{sec:Finite radius scenario}

Let \(\Omega(t)\subset\mathbb{R}^2\) be a bounded tumor domain with a
sufficiently regular boundary. The free-boundary problem
\eqref{eqn:limitpressure}--\eqref{eqn:invivonutrient} is formulated for
general evolving domains. Recall that, for a patch solution,
the right-hand side of \eqref{eqn:pressure_equation} is
\(c^{\mathrm{(i)}}/c_B\), where the subscript \(\infty\) is omitted as
stated at the end of the previous section. For the finite-radius
stability analysis, we restrict attention to boundaries that can be
represented as radial graphs about the origin:
\begin{align}
\label{eqn: boundary_at_Rtheta}
\Omega(t)
&:=\left\{(r,\theta)\,\middle|\,
0\leq r<R(\theta,t),\ 0\leq\theta<2\pi\right\},\\
\mathcal{B}(t)
&:=\partial\Omega(t)
=\left\{(r,\theta)\,\middle|\,
r=R(\theta,t),\ 0\leq\theta<2\pi\right\},
\end{align}
where \(R(\theta,t)>0\) is \(2\pi\)-periodic in \(\theta\). The
corresponding exterior domain is denoted by
\(\Omega^c(t):=\mathbb{R}^2\setminus\overline{\Omega(t)}\).

\subsubsection{Symmetric solutions}

In the radially symmetric case, the tumor domain is a disk of radius
\(R(t)>0\), and its boundary is the corresponding circle. At each fixed
time, we set \(R=R(t)\) and suppress the time dependence of the domains
and the associated radial solutions. Thus,
\(\Omega_R:=\{(r,\theta)\mid 0\leq r<R,\ 0\leq\theta<2\pi\}\)
and
\(\mathcal{B}_R:=\{(r,\theta)\mid r=R,\ 0\leq\theta<2\pi\}\).
In accordance with the notation above, we set
\(\Omega_R^c:=\mathbb{R}^2\setminus\overline{\Omega_R}\).
The nutrient system \eqref{eqn:invivonutrient} then becomes
\begin{subequations}
\label{eqn: nutrients_radius}
\begin{alignat}{4}
\label{eqn: c_inside_radius}
-\frac{1}{r}\partial_r(r\partial_r c^{\mathrm{(i)}})
+\lambda c^{\mathrm{(i)}}
&=0,
&\quad&\text{in}
&\quad&\Omega_R,\\
\label{eqn: c_outside_radius}
-\frac{1}{r}\partial_r(r\partial_r c^{\mathrm{(o)}})
+c^{\mathrm{(o)}}
&=c_B,
&\quad&\text{in}
&\quad&\Omega_R^c,\\
\label{eqn: boundary_continuity_radius}
c^{\mathrm{(i)}}
&=c^{\mathrm{(o)}},
&\quad&\text{on}
&\quad&\mathcal{B}_R,\\
\label{eqn: boundary_derivative_continuity_radius}
\partial_r c^{\mathrm{(i)}}
&=\partial_r c^{\mathrm{(o)}},
&\quad&\text{on}
&\quad&\mathcal{B}_R.
\end{alignat}
\end{subequations}
The radial nutrient solutions also satisfy
\[
\partial_r c^{\mathrm{(i)}}(0)=0,
\qquad
c^{\mathrm{(o)}}(r)\longrightarrow c_B
\quad\text{as }r\to+\infty.
\]

The pressure equation \eqref{eqn:pressure_equation} and the
boundary condition \eqref{bc:p_zero} reduce to
\begin{subequations}
\label{eqn: pressure_radius}
\begin{alignat}{4}
-\frac{1}{r}\partial_r(r\partial_r p)+\gamma p
&=\frac{c^{\mathrm{(i)}}}{c_B},
&\quad&\text{in}
&\quad&\Omega_R,\\
\label{eqn:pressure_boundary_radius}
p&=0,
&\quad&\text{on}
&\quad&\mathcal{B}_R,
\end{alignat}
\end{subequations}
together with the regularity condition \(\partial_r p(0)=0\).

Under the radial symmetry assumption, the nutrient equations can be
solved explicitly. We denote the resulting radial nutrient
concentrations by $c_0^{\mathrm{(i)}}$ and
$c_0^{\mathrm{(o)}}$:
\begin{subequations}
\label{eqn: nutrient_solution_radius}
\begin{alignat}{2}
c_0^{\mathrm{(i)}}(r)
&=c_Ba_0(R)I_0(\sqrt{\lambda}r),
&\quad&0\leq r\leq R,\\
c_0^{\mathrm{(o)}}(r)
&=c_B\bigl(1+b_0(R)K_0(r)\bigr),
&\quad&r\geq R,
\end{alignat}
\end{subequations}
where
\begin{subequations}
\label{eqn: radial_nutrient_coefficients}
\begin{alignat}{2}
a_0(R)
&=\frac{K_1(R)}
{\sqrt{\lambda}K_0(R)I_0'(\sqrt{\lambda}R)
-K_0'(R)I_0(\sqrt{\lambda}R)}
=\frac{K_1(R)}{\mathcal{D}_0(R)},\\
b_0(R)
&=-\frac{\sqrt{\lambda}I_1(\sqrt{\lambda}R)}
{\sqrt{\lambda}K_0(R)I_0'(\sqrt{\lambda}R)
-K_0'(R)I_0(\sqrt{\lambda}R)}
=-\frac{\sqrt{\lambda}I_1(\sqrt{\lambda}R)}{\mathcal{D}_0(R)}.
\end{alignat}
\end{subequations}
Here \(I_j\) and \(K_j\) denote the modified Bessel functions of the
first and second kinds, respectively, a prime denotes differentiation
with respect to the argument, and
\begin{equation}
\label{eqn: D_i_definition}
\mathcal{D}_i(R)
:=
\sqrt{\lambda}K_i(R)I_i'(\sqrt{\lambda}R)
-K_i'(R)I_i(\sqrt{\lambda}R),
\qquad i=0,1,2,\ldots.
\end{equation}

Substituting \(c_0^{\mathrm{(i)}}\) into
\eqref{eqn: pressure_radius} and solving the resulting boundary value
problem, we obtain the radially symmetric pressure solution \(p_0\):
\begin{equation}
\label{eqn: pressure_solution_radius}
p_0(r)
=\begin{cases}
\displaystyle
\frac{a_0(R)}{\lambda-\gamma}
\left[
\frac{I_0(\sqrt{\lambda}R)}{I_0(\sqrt{\gamma}R)}
I_0(\sqrt{\gamma}r)-I_0(\sqrt{\lambda}r)
\right],
& \lambda\neq\gamma,\\[2ex]
\displaystyle
\frac{a_0(R)}{2\sqrt{\lambda}}
\left[
\frac{R I_1(\sqrt{\lambda}R)}{I_0(\sqrt{\lambda}R)}
I_0(\sqrt{\lambda}r)-rI_1(\sqrt{\lambda}r)
\right],
& \lambda=\gamma,
\end{cases}
\qquad \text{in }\Omega_R.
\end{equation}
The second expression in \eqref{eqn: pressure_solution_radius} is the
limit of the first one as \(\gamma\to\lambda\).
Finally, by the boundary velocity law
\eqref{eqn:boundary_speed}, the unperturbed radius evolves according to
\begin{equation}
\label{eqn: radial_radius_evolution}
\frac{dR}{dt}=-\partial_r p_0(R).
\end{equation}
This completes the construction of the radially symmetric
solution. It will serve as the base state for the perturbation analysis
below.

\subsubsection{Single-mode perturbation}
\label{sec:Single mode perturbation_radial}
In this subsection, we determine the first-order nutrient
and pressure responses to a single Fourier-mode perturbation of the
radially symmetric tumor boundary. The case \(\gamma=0\) reduces to the
model without pressure feedback studied in \cite{feng2023tumor}; its
finite-radius linear stability analysis was carried out there. We therefore
restrict the present subsection to \(\gamma>0\). We first treat the
nonresonant case \(\lambda\neq\gamma\), and then give the corresponding
formula for \(\lambda=\gamma\).

We write a small perturbation of the radial boundary as
\(R_\epsilon(\theta,t)=R(t)+\epsilon(t)\mathcal{P}(\theta)\), where
\(\epsilon(t)\) is its amplitude and \(\mathcal{P}(\theta)\) is a fixed,
sufficiently regular \(2\pi\)-periodic profile. Expanding
\(\mathcal{P}\) into its Fourier series and linearizing about the disk,
the individual Fourier modes decouple. Moreover, by rotational
invariance, the modes \(\cos(l\theta)\) and \(\sin(l\theta)\) have the
same growth rate for each \(l\geq1\). Therefore, after a rotation of the
angular coordinate, it suffices to analyze a single cosine mode. Taking
\(\mathcal{P}(\theta)=\cos(l\theta)\), we denote the corresponding
amplitude by \(\epsilon_l(t)\) and set
\begin{equation}
\label{eqn: single_mode_boundary_radius}
R_{\epsilon,l}(\theta,t)
:=R(t)+\epsilon_l(t)\cos(l\theta),
\qquad l=1,2,\ldots,
\qquad \frac{|\epsilon_l(t)|}{R(t)}\ll1.
\end{equation}
The mode \(l=0\), which changes only the radius of the disk, is already
incorporated into \(R(t)\). The mode \(l=1\) represents an infinitesimal
translation, while genuine shape perturbations correspond to
\(l\geq2\).

We denote the perturbed tumor domain and its boundary by
\begin{align}
\label{eqn: region_perturbed_radius}
\Omega_{\epsilon,l}(t)
&:=\left\{(r,\theta)\,\middle|\,
0\leq r<R_{\epsilon,l}(\theta,t),\ 0\leq\theta<2\pi\right\},\\
\label{eqn: boundary_perturbed_radius}
\mathcal{B}_{\epsilon,l}(t)
&:=\partial\Omega_{\epsilon,l}(t)
=\left\{(r,\theta)\,\middle|\,
r=R_{\epsilon,l}(\theta,t),\ 0\leq\theta<2\pi\right\}.
\end{align}
As above, the corresponding exterior domain is denoted by
\(\Omega_{\epsilon,l}^c(t)\).

At each fixed time, we regard \(R=R(t)\) as a parameter. The nutrient
and pressure fields associated with \(\Omega_{\epsilon,l}(t)\) are expanded
as
\begin{subequations}
\label{eqn:asymptotic_solutions_radius}
\begin{alignat}{1}
c_{\epsilon,l}^{\mathrm{(i)}}(r,\theta,t)
&=c_0^{\mathrm{(i)}}(r)
+\epsilon_l(t)\hat c_{1,l}^{\mathrm{(i)}}(r)\cos(l\theta)
+O(\epsilon_l^2),\\
c_{\epsilon,l}^{\mathrm{(o)}}(r,\theta,t)
&=c_0^{\mathrm{(o)}}(r)
+\epsilon_l(t)\hat c_{1,l}^{\mathrm{(o)}}(r)\cos(l\theta)
+O(\epsilon_l^2),\\
p_{\epsilon,l}(r,\theta,t)
&=p_0(r)
+\epsilon_l(t)\hat p_{1,l}(r)\cos(l\theta)
+O(\epsilon_l^2).
\end{alignat}
\end{subequations}
The zeroth-order terms are the radially symmetric solutions
\(c_0^{\mathrm{(i)}}\), \(c_0^{\mathrm{(o)}}\), and \(p_0\) given in
\eqref{eqn: nutrient_solution_radius} and
\eqref{eqn: pressure_solution_radius}, while the first-order terms
describe the response to the boundary perturbation. Substitution into
the field equations \eqref{eqn:pressure_equation}, \eqref{eqn:c_in}, and
\eqref{eqn:c_out}, followed by collection of the terms of order
\(\epsilon_l\), gives
\begin{subequations}
\label{eqn:first_order_equation_radius}
\begin{alignat}{2}
-\partial_r^2\hat c_{1,l}^{\mathrm{(i)}}
-\frac{1}{r}\partial_r\hat c_{1,l}^{\mathrm{(i)}}
+\left(\frac{l^2}{r^2}+\lambda\right)
\hat c_{1,l}^{\mathrm{(i)}}
&=0,\\
-\partial_r^2\hat c_{1,l}^{\mathrm{(o)}}
-\frac{1}{r}\partial_r\hat c_{1,l}^{\mathrm{(o)}}
+\left(\frac{l^2}{r^2}+1\right)
\hat c_{1,l}^{\mathrm{(o)}}
&=0,\\
-\partial_r^2\hat p_{1,l}
-\frac{1}{r}\partial_r\hat p_{1,l}
+\left(\frac{l^2}{r^2}+\gamma\right)\hat p_{1,l}
&=\frac{\hat c_{1,l}^{\mathrm{(i)}}}{c_B}.
\end{alignat}
\end{subequations}

The outward unit normal to
\(\mathcal{B}_{\epsilon,l}(t)\) satisfies
\(n_{\epsilon,l}=e_r+O(\epsilon_l)\), with its first-order correction in
the tangential \(e_\theta\)-direction, where \(e_r\) and \(e_\theta\)
are the polar unit vectors. Using
\eqref{eqn:asymptotic_solutions_radius}, Taylor expansion of the nutrient
values and normal derivatives on \(\mathcal{B}_{\epsilon,l}(t)\), together
with the transmission conditions
\eqref{bc:c_continuity}--\eqref{bc:flux_continuity} and their zero-order
counterparts \eqref{eqn: boundary_continuity_radius}--
\eqref{eqn: boundary_derivative_continuity_radius}, yields the following
boundary conditions for the first-order nutrient terms. A detailed derivation can
be found in \cite[Section 3.3]{feng2023tumor}.
\begin{subequations}
\label{eqn:first_order_boundary_condition_radius}
\begin{alignat}{2}
\hat c_{1,l}^{\mathrm{(i)}}(R)
&=\hat c_{1,l}^{\mathrm{(o)}}(R),\\
\partial_r^2c_0^{\mathrm{(i)}}(R)
+\partial_r\hat c_{1,l}^{\mathrm{(i)}}(R)
&=\partial_r^2c_0^{\mathrm{(o)}}(R)
+\partial_r\hat c_{1,l}^{\mathrm{(o)}}(R).
\end{alignat}
Applying the same expansion to the pressure boundary
condition \eqref{bc:p_zero}, with \(\partial\Omega(t)\) replaced by
\(\mathcal{B}_{\epsilon,l}(t)\), and using \(p_0(R)=0\), gives
\begin{equation}
\hat p_{1,l}(R)=-\partial_rp_0(R).
\end{equation}
\end{subequations}
The interior nutrient perturbation is bounded at \(r=0\),
while the exterior nutrient perturbation satisfies
\(\hat c_{1,l}^{\mathrm{(o)}}(r)\to0\) as \(r\to+\infty\). The
solutions of the first two equations in
\eqref{eqn:first_order_equation_radius} therefore take the form
\begin{subequations}
\label{eqn:first_order_nutrient_solution_radius}
\begin{alignat}{2}
\hat c_{1,l}^{\mathrm{(i)}}(r)
&=-c_Ba_{1,l}(R)I_l(\sqrt{\lambda}r),\\
\hat c_{1,l}^{\mathrm{(o)}}(r)
&=-c_Bb_{1,l}(R)K_l(r).
\end{alignat}
\end{subequations}
Recalling the definition of \(\mathcal{D}_i(R)\) in
\eqref{eqn: D_i_definition}, the zero-order solutions satisfy
\[
\partial_r^2c_0^{\mathrm{(o)}}(R)
-\partial_r^2c_0^{\mathrm{(i)}}(R)
=-c_B\sqrt{\lambda}
\frac{\mathcal{D}_1(R)}{\mathcal{D}_0(R)}.
\]
Consequently, the first two conditions in
\eqref{eqn:first_order_boundary_condition_radius} yield
\begin{subequations}
\label{eqn:first_order_nutrient_coefficients_radius}
\begin{alignat}{2}
a_{1,l}(R)
&=\frac{\sqrt{\lambda}\mathcal{D}_1(R)K_l(R)}
{\mathcal{D}_0(R)\mathcal{D}_l(R)},\\
b_{1,l}(R)
&=\frac{\sqrt{\lambda}\mathcal{D}_1(R)
I_l(\sqrt{\lambda}R)}
{\mathcal{D}_0(R)\mathcal{D}_l(R)}.
\end{alignat}
\end{subequations}

In the nonresonant case \(\lambda\neq\gamma\), the bounded solution of the
pressure equation in \eqref{eqn:first_order_equation_radius} can be
written as
\begin{equation}
\label{eqn:first_order_pressure_solution_radius}
\hat p_{1,l}(r)
=A_{1,l}(R)I_l(\sqrt{\gamma}r)
+B_{1,l}(R)I_l(\sqrt{\lambda}r),
\end{equation}
where
\begin{subequations}
\label{eqn:first_order_pressure_coefficients_radius}
\begin{alignat}{2}
B_{1,l}(R)
&=\frac{a_{1,l}(R)}{\lambda-\gamma},\\
A_{1,l}(R)
&=-\frac{\partial_rp_0(R)
+B_{1,l}(R)I_l(\sqrt{\lambda}R)}
{I_l(\sqrt{\gamma}R)}.
\end{alignat}
\end{subequations}
Here \(\partial_rp_0(R)\) is determined by
\eqref{eqn: pressure_solution_radius}; explicitly,
\begin{equation}
\label{eqn:base_pressure_derivative_radius}
\partial_rp_0(R)
=\frac{a_0(R)}{\lambda-\gamma}
\left[
\frac{I_0(\sqrt{\lambda}R)}{I_0(\sqrt{\gamma}R)}
\sqrt{\gamma}I_1(\sqrt{\gamma}R)
-\sqrt{\lambda}I_1(\sqrt{\lambda}R)
\right].
\end{equation}

When \(\lambda=\gamma\), the singularity in
\eqref{eqn:first_order_pressure_coefficients_radius} is removable. The
nutrient perturbations and their coefficients remain unchanged, while
taking the limit \(\gamma\to\lambda\) in the nonresonant pressure solution
gives
\begin{equation}
\label{eqn:first_order_pressure_solution_resonant_radius}
\begin{aligned}
\hat p_{1,l}(r)
={}&-\partial_rp_0(R)
\frac{I_l(\sqrt{\lambda}r)}{I_l(\sqrt{\lambda}R)}\\
&+\frac{a_{1,l}(R)}{2\sqrt{\lambda}}
\left[
rI_l'(\sqrt{\lambda}r)
-R I_l'(\sqrt{\lambda}R)
\frac{I_l(\sqrt{\lambda}r)}{I_l(\sqrt{\lambda}R)}
\right],
\end{aligned}
\end{equation}
where
\begin{equation}
\label{eqn:base_pressure_derivative_resonant_radius}
\partial_rp_0(R)
=\frac{a_0(R)R}{2}
\left[
\frac{I_1(\sqrt{\lambda}R)^2}{I_0(\sqrt{\lambda}R)}
-I_0(\sqrt{\lambda}R)
\right].
\end{equation}

This completes the determination of the first-order
correction terms for the coupled system.

\subsubsection{Evolution equation for the boundary perturbation}
\label{sec:perturbation_growth_rate_radius}

Using the first-order pressure obtained above, we next derive the
evolution equation for the boundary perturbation amplitude.
First, observe that the normal velocity of the radial graph
\eqref{eqn: single_mode_boundary_radius} satisfies
\begin{equation}
\label{eqn:normal_velocity_expansion_radius}
V_{\epsilon,l}
=
\frac{\partial_tR_{\epsilon,l}}
{\sqrt{1+R_{\epsilon,l}^{-2}
(\partial_\theta R_{\epsilon,l})^2}}
=
\frac{dR}{dt}
+\frac{d\epsilon_l}{dt}\cos(l\theta)
+O(\epsilon_l^2).
\end{equation}
On the other hand, applying Taylor expansion to evaluate the
pressure gradient in \eqref{eqn:asymptotic_solutions_radius} on the
perturbed boundary \(\mathcal{B}_{\epsilon,l}(t)\), and using
\(n_{\epsilon,l}=e_r+O(\epsilon_l)\), we obtain
\begin{equation}
\label{eqn:normal_pressure_expansion_radius}
\nabla p_{\epsilon,l}\cdot n_{\epsilon,l}
=
\partial_rp_0(R)
+\epsilon_l(t)
\left[
\partial_r^2p_0(R)+\partial_r\hat p_{1,l}(R)
\right]\cos(l\theta)
+O(\epsilon_l^2)
\end{equation}
on \(\mathcal{B}_{\epsilon,l}(t)\). The angular component of the pressure
gradient contributes only at order \(O(\epsilon_l^2)\). Substituting
\eqref{eqn:normal_velocity_expansion_radius} and
\eqref{eqn:normal_pressure_expansion_radius} into the boundary condition
\(V_{\epsilon,l}=-\nabla p_{\epsilon,l}\cdot n_{\epsilon,l}\), and then
using \eqref{eqn: radial_radius_evolution}, yields the following
amplitude equation.
\begin{equation}
\label{eqn:growth_rate_definition_radius}
\epsilon_l^{-1}\frac{d\epsilon_l}{dt}
=\underbrace{-\left[
\partial_r^2p_0(R(t))+\partial_r\hat p_{1,l}(R(t))
\right]}_{=: \mathcal{F}_l(R(t);\lambda,\gamma)}
+O(\epsilon_l).
\end{equation}
Following the terminology of Cristini, Lowengrub, and Nie
\cite{cristini2003nonlinear}, we call \(\mathcal{F}_l\) the boundary
evolution function of the \(l\)-th mode.
Thus, \(\mathcal{F}_l<0\) corresponds to instantaneous decay of the
\(l\)-th mode, whereas \(\mathcal{F}_l>0\) corresponds to instantaneous
growth.

For \(\lambda\neq\gamma\), we substitute the expressions for \(p_0\)
and \(\hat p_{1,l}\) from
\eqref{eqn: pressure_solution_radius} and
\eqref{eqn:first_order_pressure_solution_radius}--
\eqref{eqn:first_order_pressure_coefficients_radius} into
\eqref{eqn:growth_rate_definition_radius} and obtain
\begin{align}
\label{eqn:growth_rate_nonresonant_radius}
\mathcal{F}_l(R;\lambda,\gamma)
={}&
a_0(R)I_0(\sqrt{\lambda}R)
+\left[
\frac{1}{R}
+\sqrt{\gamma}
\frac{I_l'(\sqrt{\gamma}R)}
{I_l(\sqrt{\gamma}R)}
\right]\partial_rp_0(R)
\\
&+
\frac{a_{1,l}(R)I_l(\sqrt{\lambda}R)}
{\lambda-\gamma}
\left[
\sqrt{\gamma}
\frac{I_l'(\sqrt{\gamma}R)}
{I_l(\sqrt{\gamma}R)}
-\sqrt{\lambda}
\frac{I_l'(\sqrt{\lambda}R)}
{I_l(\sqrt{\lambda}R)}
\right].
\nonumber
\end{align}
Together with
\eqref{eqn:first_order_nutrient_coefficients_radius} and
\eqref{eqn:base_pressure_derivative_radius}, formula
\eqref{eqn:growth_rate_nonresonant_radius} is explicit in
\(R,\lambda,\gamma\), and \(l\).

For \(\gamma=\lambda\), a similar substitution of the corresponding expressions in
\eqref{eqn: pressure_solution_radius} and
\eqref{eqn:first_order_pressure_solution_resonant_radius} into
\eqref{eqn:growth_rate_definition_radius} gives the resonant growth rate
\(\mathcal{F}_l(R;\lambda,\lambda)\). One can verify that
\begin{equation}
\label{eqn:growth_rate_resonant_limit_radius}
\mathcal{F}_l(R;\lambda,\lambda)
=
\lim_{\gamma\to\lambda}
\mathcal{F}_l(R;\lambda,\gamma).
\end{equation}

Since \(R(t)\) evolves, \(\mathcal{F}_l\) is an instantaneous rather than
a constant growth rate. At the linearized level,
\begin{equation}
\label{eqn:amplitude_integral_radius}
\epsilon_l(t)
=
\epsilon_l(0)
\exp\left\{
\int_0^t
\mathcal{F}_l(R(s);\lambda,\gamma)\,ds
\right\}.
\end{equation}
Parameterizing the perturbation amplitude by the unperturbed radius
\(R\) and using
\eqref{eqn: radial_radius_evolution},
\eqref{eqn:amplitude_integral_radius} is equivalent to
\begin{equation}
\label{eqn:amplitude_radius_integral_radius}
\epsilon_l(R)
=
\epsilon_l(R_0)
\exp\left\{
\int_{R_0}^{R}
\frac{\mathcal{F}_l(\rho;\lambda,\gamma)}
{-\partial_rp_0(\rho)}\,d\rho
\right\}.
\end{equation}
Here \(R_0:=R(0)\) is the initial radius, and
\(\epsilon_l(R)\) denotes the perturbation amplitude evaluated when the
unperturbed radius reaches \(R\).

Combining \eqref{eqn: radial_radius_evolution} and
\eqref{eqn:amplitude_radius_integral_radius}, we can reconstruct the
linearized boundary evolution and examine its stability, as shown in
Figure \ref{fig:finite_radius_boundary_evolution}. The figure shows two
effects of pressure feedback: it can turn a stable mode into an
unstable one or make an unstable perturbation grow faster.

At \(\gamma=0\), we use \(\mathcal F_l(R;\lambda,0)\) for the boundary
evolution function of the model without pressure feedback, equivalently its
fixed-\(R\) limit as \(\gamma\to0^+\). The computations in
\cite{feng2023tumor} indicate that
\(\mathcal F_l(R;\lambda,0)<0\) for every \(R>0\), \(0<\lambda\leq1\),
and \(l\geq2\). In the left-hand plot, we take \(l=5\), \(\lambda=1\),
\(4\leq R\leq8\), and \(\epsilon_5(R_0)/R_0=0.07\). In this case,
\(\mathcal F_5(R;1,0)<0<\mathcal F_5(R;1,5)\), and the ratios
\(\epsilon_5(8)/\epsilon_5(4)\) are \(0.35\) and \(1.71\) for
\(\gamma=0\) and \(\gamma=5\), respectively. In the right-hand plot, we
take \(l=5\), \(\lambda=100\), \(1.5\leq R\leq4\), and
\(\epsilon_5(R_0)/R_0=0.03\). The ratios
\(\epsilon_5(4)/\epsilon_5(1.5)\) are \(4.11\) and \(7.69\) for
\(\gamma=0\) and \(\gamma=5\), respectively, so the perturbation grows
faster in the latter case.

\begin{center}
\centering
\includegraphics[width=\textwidth]{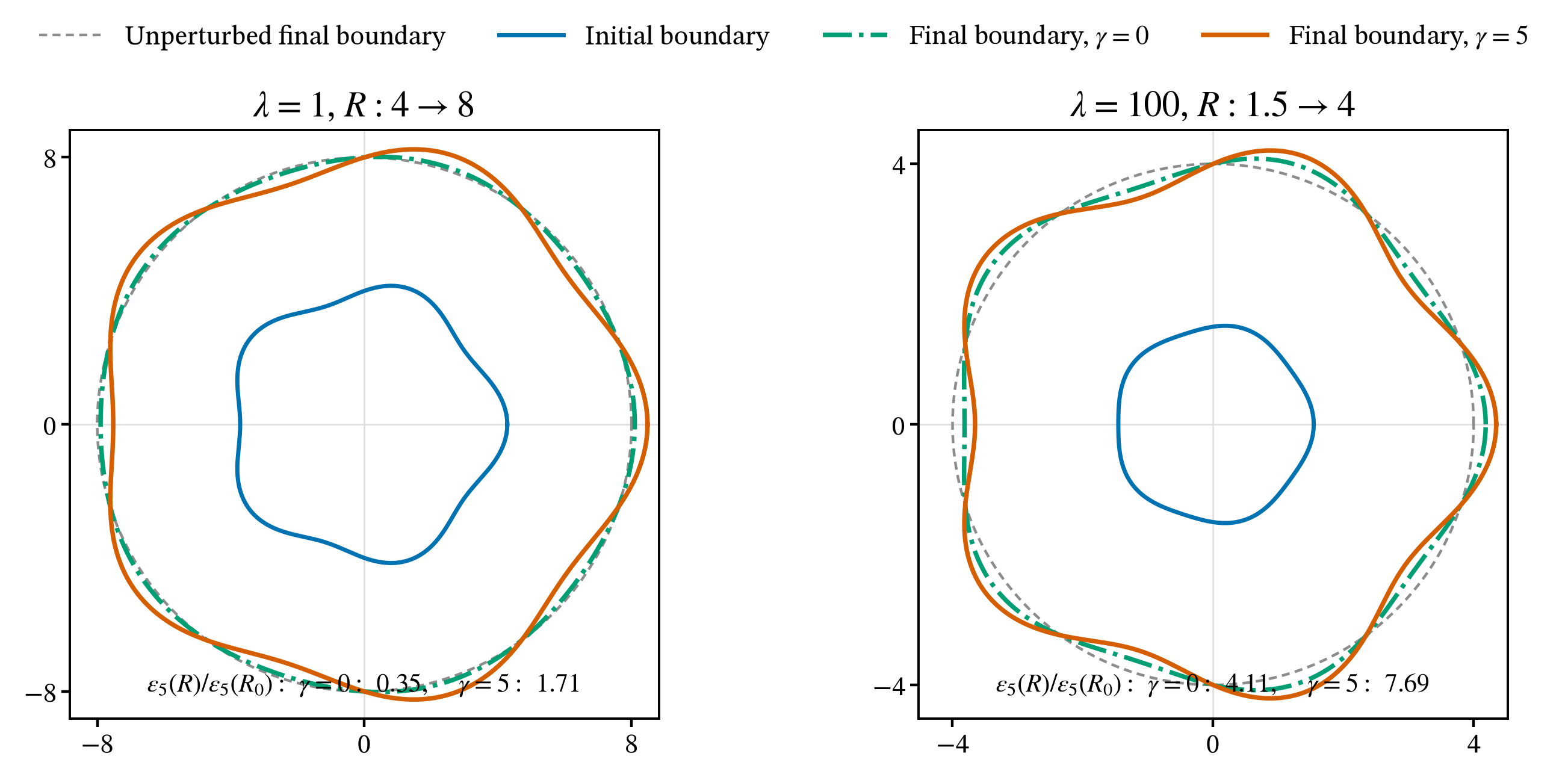}
\captionof{figure}{{Linearized evolution of the \(l=5\) boundary
perturbation. The left-hand plot uses \(\epsilon_5(R_0)/R_0=0.07\),
\(\lambda=1\), \(R_0=4\), and \(R=8\); the right-hand plot uses
\(\epsilon_5(R_0)/R_0=0.03\), \(\lambda=100\), \(R_0=1.5\), and \(R=4\).
The blue curve is the initial boundary in each plot, the dashed gray circle
is the unperturbed final boundary, and the green and orange curves are the
final boundaries for \(\gamma=0\) and \(\gamma=5\), respectively. The
ratios \(\epsilon_5(R)/\epsilon_5(R_0)\) shown in the plots are computed from
\eqref{eqn:amplitude_radius_integral_radius}.}}
\label{fig:finite_radius_boundary_evolution}
\end{center}

\subsubsection{Boundary stability and parameter dependence}
\label{sec:boundary_stability_parameter_radius}

In this subsection, we further characterize the boundary
instability of the radially symmetric solution by establishing several
mathematical properties of \(\mathcal{F}_l\).

\begin{prop}
\label{prop:growth_rate_asymptotics_radius}
For fixed \(\lambda>0\) and \(\gamma>0\), the translation
mode is neutral:
\begin{equation}
\label{eqn:translation_mode_neutral_radius}
\mathcal{F}_1(R;\lambda,\gamma)=0
\qquad\text{for every }R>0.
\end{equation}
For each integer \(l\geq2\),
\begin{subequations}
\label{eqn:growth_rate_asymptotics_radius}
\begin{align}
\lim_{R\to0^+}\mathcal{F}_l(R;\lambda,\gamma)
&=\frac{1-l}{2}<0,
\\
\mathcal{F}_l(R;\lambda,\gamma)
&=
\frac{l^2-1}
{2(\sqrt{\lambda}+1)(\sqrt{\lambda}+\sqrt{\gamma})R^2}
+O(R^{-3})
\qquad\text{as }R\to+\infty.
\end{align}
\end{subequations}
For fixed \(R>0\), the high-frequency behavior is
\begin{equation}
\label{eqn:growth_rate_high_mode_radius}
\mathcal{F}_l(R;\lambda,\gamma)
=
\frac{\partial_rp_0(R)}{R}\,l+O(1)
\qquad\text{as }l\to+\infty.
\end{equation}
Consequently, for every \(l\geq2\), the function
\(\mathcal{F}_l(\,\cdot\,;\lambda,\gamma)\) has at least one zero on
\((0,+\infty)\), while for each fixed \(R>0\), all sufficiently large
modes are stable.
\end{prop}

\begin{proof}
The identity
\eqref{eqn:translation_mode_neutral_radius} follows directly from
\eqref{eqn:growth_rate_nonresonant_radius}, together with
\eqref{eqn:growth_rate_resonant_limit_radius} when
\(\lambda=\gamma\).

For \(l\geq2\), the small-radius limit follows by substituting the
standard expansions of \(I_l\) and \(K_l\) at the origin into
\eqref{eqn:growth_rate_nonresonant_radius}. For the
large-radius limit, we use the following asymptotic expansions for the
Bessel functions for fixed \(l\) and \(\lambda>0\):
\begin{subequations}
\label{eqn:bessel_large_radius_fixed_mode}
\begin{align}
\sqrt{\lambda}\,
\frac{I_l'(\sqrt{\lambda}R)}{I_l(\sqrt{\lambda}R)}
&=
\sqrt{\lambda}-\frac{1}{2R}
+\frac{4l^2-1}{8\sqrt{\lambda}R^2}
+O(R^{-3}),
\\
-\frac{K_l'(R)}{K_l(R)}
&=
1+\frac{1}{2R}
+\frac{4l^2-1}{8R^2}
+O(R^{-3}).
\end{align}
\end{subequations}
The first expansion also holds with \(\lambda\) replaced by
\(\gamma\). Substitution into the definitions of \(a_0\),
\(\mathcal D_l\), and \(a_{1,l}\) in
\eqref{eqn: radial_nutrient_coefficients},
\eqref{eqn: D_i_definition}, and
\eqref{eqn:first_order_nutrient_coefficients_radius}, respectively,
and then into
\eqref{eqn:growth_rate_nonresonant_radius}, gives the second expansion in
\eqref{eqn:growth_rate_asymptotics_radius}. These expansions are locally
uniform for positive \(\gamma\); hence the resonant case follows by
continuity as \(\gamma\to\lambda\). The two radius
asymptotics, together with the continuity in \(R\), imply that
\(\mathcal{F}_l\) has at least one zero for every \(l\geq2\).

We next fix \(R>0\) and consider the limit
\(l\to+\infty\). The standard large-order asymptotics of the modified
Bessel functions, together with
\eqref{eqn:first_order_nutrient_coefficients_radius}, give, for
\(\mu=\lambda,\gamma\),
\[
\sqrt{\mu}\,
\frac{I_l'(\sqrt{\mu}R)}{I_l(\sqrt{\mu}R)}
=\frac{l}{R}+O(l^{-1}),
\qquad
a_{1,l}(R)I_l(\sqrt{\lambda}R)=O(l^{-1}).
\]
These relations yield \eqref{eqn:growth_rate_high_mode_radius}; at
\(\lambda=\gamma\), the same conclusion follows from the resonant formula
\eqref{eqn:first_order_pressure_solution_resonant_radius}. Since
\(\partial_rp_0(R)<0\), sufficiently large modes have negative growth
rate.
\end{proof}

\begin{rem}
Proposition \ref{prop:growth_rate_asymptotics_radius} guarantees at
least one zero of \(\mathcal{F}_l(\,\cdot\,;\lambda,\gamma)\) for each
\(l\geq2\). Numerical evaluation over the parameter ranges considered
here suggests that this zero is in fact unique: the curve crosses the
horizontal axis only once as \(R\) increases, as illustrated in Figure
\ref{fig:finite_radius_growth_rates}. Establishing this uniqueness
analytically is considerably more difficult because of the coupled
Bessel-function structure of \eqref{eqn:growth_rate_nonresonant_radius},
and we do not pursue it here.
\end{rem}

Figure \ref{fig:finite_radius_growth_rates} illustrates the conclusions
of Proposition \ref{prop:growth_rate_asymptotics_radius} in the resonant
case \(\lambda=\gamma=1\).

\begin{center}
\begin{minipage}{\textwidth}
\centering
\includegraphics[width=\textwidth]{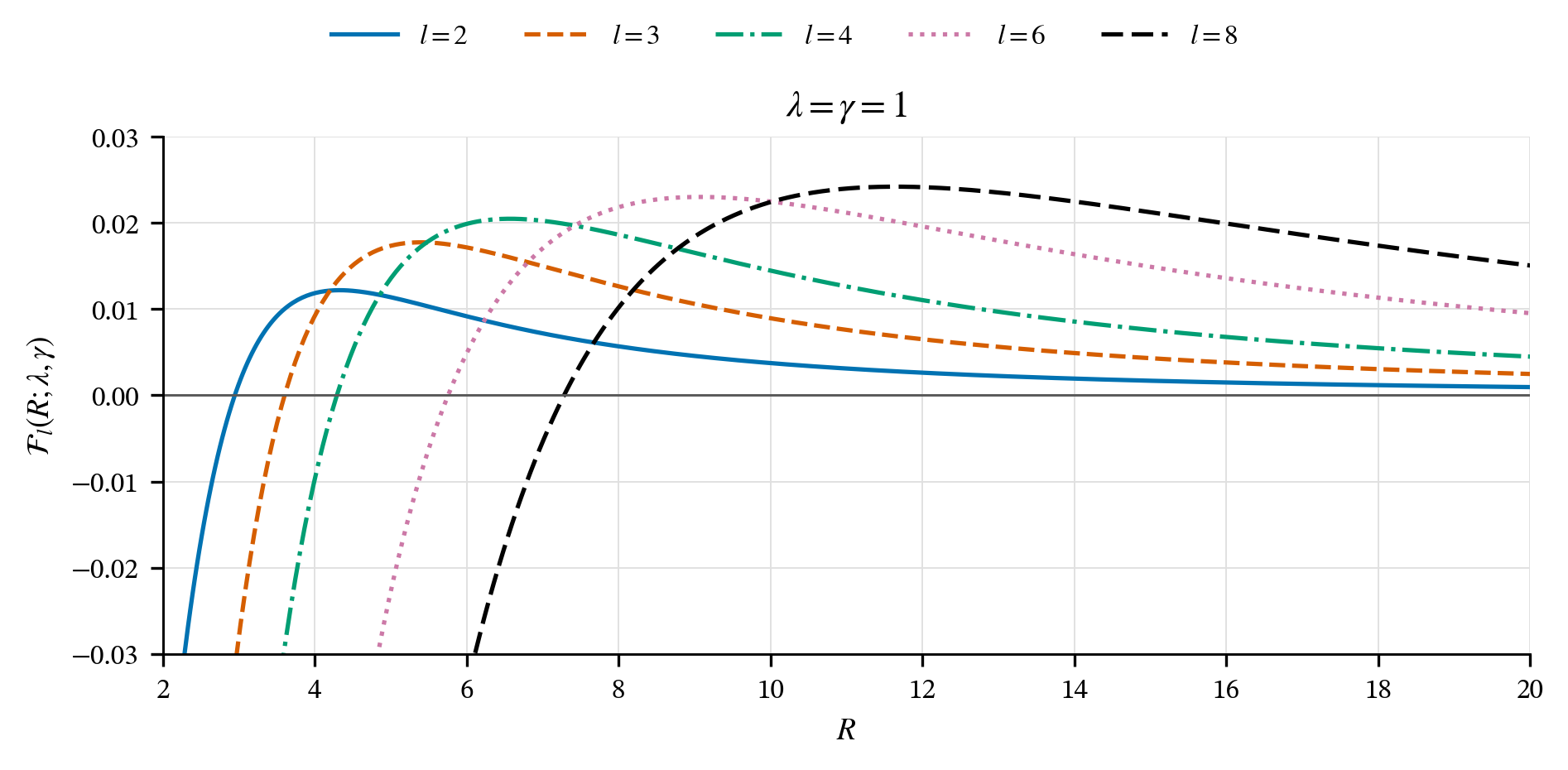}
\captionof{figure}{The growth rate
\(\mathcal{F}_l(R;1,1)\) for \(l=2,3,4,6,8\) in a neighborhood of the
stability thresholds. Each displayed mode crosses the zero level,
reaches a positive maximum, and then decreases toward zero as \(R\)
increases.}
\label{fig:finite_radius_growth_rates}
\end{minipage}
\end{center}

Since Proposition \ref{prop:growth_rate_asymptotics_radius} assumes
\(\gamma>0\), we next compare its large-radius behavior with the
endpoint \(\gamma=0\), where pressure feedback is absent.

\begin{prop}
\label{prop:noncommuting_limits_radius}
Let \(\lambda>0\) and \(l\geq2\). Then the two iterated limits satisfy
\begin{subequations}
\label{eqn:noncommuting_limits_radius}
\begin{align}
\lim_{\gamma\to0^+}\lim_{R\to+\infty}
R^2\mathcal{F}_l(R;\lambda,\gamma)
&=
\frac{l^2-1}
{2\sqrt{\lambda}(\sqrt{\lambda}+1)},
\\
\lim_{R\to+\infty}\lim_{\gamma\to0^+}
R^2\mathcal{F}_l(R;\lambda,\gamma)
&=
\frac{(l^2-1)(\sqrt{\lambda}-1)}
{2\lambda(\sqrt{\lambda}+1)}.
\end{align}
\end{subequations}
\end{prop}

\begin{proof}
For fixed \(\gamma>0\), the large-radius expansion
\eqref{eqn:growth_rate_asymptotics_radius} directly gives the first
identity in \eqref{eqn:noncommuting_limits_radius}. To compute the limits
in the reverse order, we first take \(\gamma\to0^+\) in the exact expression
\eqref{eqn:growth_rate_nonresonant_radius}. For fixed \(R>0\),
\[
\sqrt{\gamma}\,
\frac{I_l'(\sqrt{\gamma}R)}{I_l(\sqrt{\gamma}R)}
\longrightarrow \frac{l}{R}
\qquad\text{as }\gamma\to0^+.
\]
The resulting expression is the boundary evolution function
without pressure feedback derived in \cite{feng2023tumor}. Substituting the large-radius
Bessel expansions \eqref{eqn:bessel_large_radius_fixed_mode} into the
formulas defining \(a_0(R)\), \(a_{1,l}(R)\), and the limiting boundary
evolution function gives
\[
\mathcal{F}_l(R;\lambda,0)
=
\frac{(l^2-1)(\sqrt{\lambda}-1)}
{2\lambda(\sqrt{\lambda}+1)R^2}
+O(R^{-3}).
\]
The second identity in \eqref{eqn:noncommuting_limits_radius} follows.
\end{proof}

\begin{rem}
\label{rem:nonuniform_gamma_radius}
For every \(l\geq2\), the first limit is positive for
all \(\lambda>0\), whereas the sign of the second limit depends on
\(\lambda\). Thus, for any fixed \(\gamma>0\), each mode eventually
becomes unstable as the tumor radius increases. At \(\gamma=0\), by
contrast, the computations in \cite{feng2023tumor} indicate that the
model without pressure feedback is stable for \(0<\lambda\leq1\) and becomes
unstable at large radii only when \(\lambda>1\). Hence, for
\(0<\lambda\leq1\), an arbitrarily small positive \(\gamma\) changes the
eventual large-radius behavior from stability to instability. This
difference shows that the convergence as \(\gamma\to0^+\) is not
uniform over arbitrarily large radii. Thus, even arbitrarily weak
pressure feedback changes the large-radius boundary stability
qualitatively.
\end{rem}

The limits above characterize the eventual behavior as
\(R\to+\infty\). To examine how pressure feedback affects the onset of
instability at finite radii, we define the following threshold for each
\(l\geq2\):
\[
R_l^*(\lambda,\gamma)
:=
\inf\left\{
R>0:\mathcal{F}_l(R;\lambda,\gamma)=0
\right\}
\]
for \(\gamma\geq0\), with the convention
\(R_l^*(\lambda,\gamma)=+\infty\) when no zero exists. Proposition
\ref{prop:growth_rate_asymptotics_radius} guarantees that this threshold
is finite for every \(\gamma>0\). At \(\gamma=0\), however, the
computations in \cite{feng2023tumor} indicate the two regimes described
in Remark \ref{rem:nonuniform_gamma_radius}. For
\(0<\lambda\leq1\), no zero is observed, and we therefore set
\(R_l^*(\lambda,0)=+\infty\). For \(\lambda>1\), a finite threshold is
found. We
therefore choose \(\lambda=1\) and \(\lambda=2\) as representative values
of these two regimes in Figure \ref{fig:critical_radius_pressure_feedback}.
To show both the behavior near \(\gamma=0\) and the larger values of
\(\gamma\), we use a smooth quasi-logarithmic horizontal scale; it is
linear at the origin and approaches a logarithmic scale away from it.

In the left-hand plot, the thresholds become unbounded as
\(\gamma\to0^+\), consistent with the qualitative change described in
Remark \ref{rem:nonuniform_gamma_radius}: positive pressure feedback
creates an instability that is absent at \(\gamma=0\). In the right-hand
plot, the curves instead approach
the finite thresholds in the absence of pressure feedback \(R_l^*(2,0)\).

The plots indicate that, for both representative values of \(\lambda\),
\(R_l^*\) decreases as \(\gamma\) increases and increases with \(l\).
Thus, within the parameter ranges considered here, stronger pressure
feedback causes instability at a smaller radius, while higher-frequency
modes require a larger tumor radius to become unstable.

\begin{figure}[!htbp]
\centering
\includegraphics[width=\textwidth]{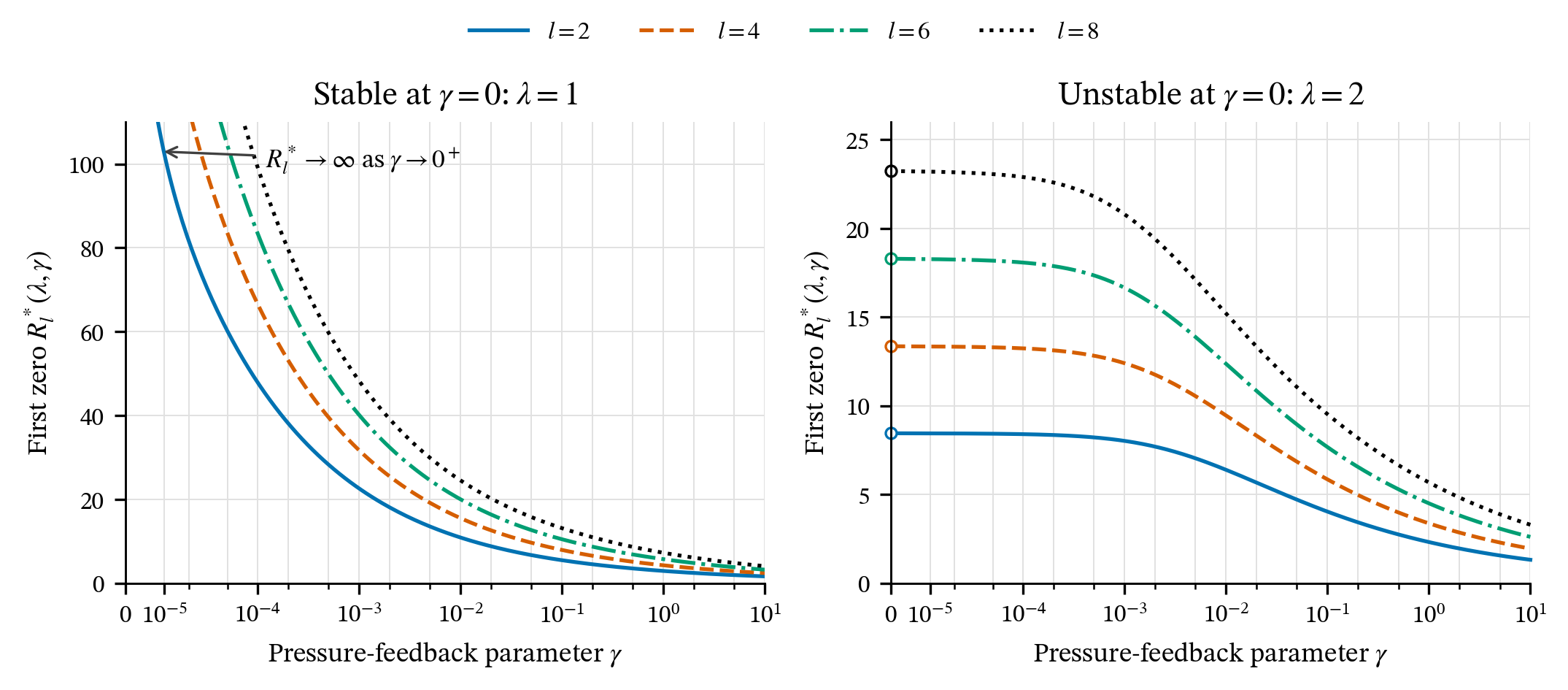}
\caption{The numerically computed threshold
\(R_l^*(\lambda,\gamma)\) for \(l=2,4,6,8\) and positive \(\gamma\) in
the range \(5\times10^{-6}\leq\gamma\leq10\). The left-hand plot uses
\(\lambda=1\),
representing the regime \(0<\lambda\leq1\), in which the thresholds
become unbounded as \(\gamma\to0^+\). The right-hand plot uses
\(\lambda=2\), representing the regime \(\lambda>1\), and also includes
\(\gamma=0\); the open circles on the vertical axis mark the finite
thresholds in the absence of pressure feedback \(R_l^*(2,0)\).}
\label{fig:critical_radius_pressure_feedback}
\end{figure}

We next consider planar traveling fronts, for which the
boundary evolution function has a simpler explicit form.

\subsection{Traveling-wave scenario}
\label{sec:traveling_wave_scenario}

This subsection is devoted to studying the boundary
instability of planar traveling fronts in a periodic tube. The analysis
follows a perturbation framework similar to that used for the
finite-radius scenario.
Let
\begin{equation}
\label{eqn: infty_tube}
\mathbb{T}_{2\pi}:=\mathbb{R}/(2\pi\mathbb{Z}),
\qquad
\mathcal{T}:=\mathbb{R}\times\mathbb{T}_{2\pi},
\end{equation}
so that \(\mathcal T\) is an infinite tube-like domain
with \(2\pi\)-periodic transverse variable \(y\). We work in
a frame translating in the \(x\)-direction with the planar front and,
for notational simplicity, continue to denote the longitudinal
coordinate in this frame by \(x\). For a
\(2\pi\)-periodic even boundary profile \(f\), define the associated
tumor region and its boundary by
\begin{align}
\label{eqn: boundary_at_fy}
\Omega_f
&:=\left\{(x,y)\in\mathcal{T}\,\middle|\,x<f(y)\right\},\\
\mathcal{B}_f
&:=\partial\Omega_f
=\left\{(x,y)\in\mathcal{T}\,\middle|\,x=f(y)\right\}.
\end{align}
The corresponding exterior region is denoted by
\(\Omega_f^c:=\mathcal{T}\setminus\overline{\Omega_f}\).

\subsubsection{Symmetric solutions}

For the planar front, \(f(y)\equiv0\), and hence
\begin{align}
\label{eqn:boundary_at_zero}
\Omega_0
&:=\left\{(x,y)\in\mathcal{T}\,\middle|\,x<0\right\},\\
\mathcal{B}_0
&:=\partial\Omega_0
=\left\{(x,y)\in\mathcal{T}\,\middle|\,x=0\right\}.
\end{align}
The nutrient concentrations depend only on \(x\) and solve
\begin{subequations}
\label{eqn: nutrients_traveling}
\begin{alignat}{4}
\label{eqn: c_inside2}
-\partial_x^2 c^{\mathrm{(i)}}+\lambda c^{\mathrm{(i)}}
&=0,
&\quad&\text{in}
&\quad&\Omega_0,\\
\label{eqn: c_outside2}
-\partial_x^2 c^{\mathrm{(o)}}+c^{\mathrm{(o)}}
&=c_B,
&\quad&\text{in}
&\quad&\Omega_0^c,\\
\label{eqn: boundary_continuity}
c^{\mathrm{(i)}}
&=c^{\mathrm{(o)}},
&\quad&\text{on}
&\quad&\mathcal{B}_0,\\
\label{eqn: boundary_derivative_continuity}
\partial_xc^{\mathrm{(i)}}
&=\partial_xc^{\mathrm{(o)}},
&\quad&\text{on}
&\quad&\mathcal{B}_0.
\end{alignat}
\end{subequations}
The far-field conditions are
\begin{equation}
\label{eqn:nutrient_far_field_traveling}
c^{\mathrm{(i)}}(x)\longrightarrow0
\quad\text{as }x\to-\infty,
\qquad
c^{\mathrm{(o)}}(x)\longrightarrow c_B
\quad\text{as }x\to+\infty.
\end{equation}
Solving this transmission problem gives
\begin{subequations}
\label{eqn: nutrient_solution_0}
\begin{alignat}{2}
\label{eqn:nutrient_solution_0_interior}
c_0^{\mathrm{(i)}}(x)
&=\frac{c_B}{\sqrt{\lambda}+1}e^{\sqrt{\lambda}x},
&\quad&x\leq0,\\
c_0^{\mathrm{(o)}}(x)
&=c_B\left(1-
\frac{\sqrt{\lambda}}{\sqrt{\lambda}+1}e^{-x}\right),
&\quad&x\geq0.
\end{alignat}
\end{subequations}

The corresponding pressure satisfies
\begin{subequations}
\label{eqn: pressure_traveling}
\begin{alignat}{4}
-\partial_x^2p_0+\gamma p_0
&=\frac{c_0^{\mathrm{(i)}}}{c_B},
&\quad&\text{in}
&\quad&\Omega_0,\\
p_0
&=0,
&\quad&\text{on}
&\quad&\mathcal{B}_0,
\end{alignat}
\end{subequations}
together with boundedness as \(x\to-\infty\).
Substituting \eqref{eqn:nutrient_solution_0_interior}
into \eqref{eqn: pressure_traveling} and solving the resulting
boundary-value problem, we obtain
\begin{equation}
\label{eqn: pressure_solution_0}
p_0(x)
=\begin{cases}
\displaystyle
\frac{e^{\sqrt{\gamma}x}-e^{\sqrt{\lambda}x}}
{(\sqrt{\lambda}+1)(\lambda-\gamma)},
& \lambda\neq\gamma,\\[2ex]
\displaystyle
-\frac{x e^{\sqrt{\lambda}x}}
{2\sqrt{\lambda}(\sqrt{\lambda}+1)},
& \lambda=\gamma,
\end{cases}
\qquad x\leq0.
\end{equation}
The second expression is the limit of the first one as
\(\gamma\to\lambda\), and both expressions satisfy \(p_0(x)>0\) for
\(x<0\). By Darcy's law, the speed of the planar front
is given by
\begin{equation}
\label{eqn: symmetric_traveling_speed}
\sigma_0
=-\partial_xp_0(0)
=\frac{1}
{(\sqrt{\lambda}+1)(\sqrt{\lambda}+\sqrt{\gamma})}.
\end{equation}

This completes the construction of the symmetric
traveling-wave solution, which serves as the base state for the
perturbation analysis below.

\subsubsection{Single-mode perturbation}
\label{sec:single_mode_perturbation_traveling}

As in the finite-radius case, it is sufficient to
consider a single cosine-mode perturbation

\begin{equation}
\label{eqn:single_mode_boundary_traveling}
f_{\epsilon,l}(y,t)
:=\epsilon_l(t)\cos(ly),
\qquad
|\epsilon_l(t)|\ll1.
\end{equation}

Here \(l\geq1\) is an integer; the zero mode only
translates the planar front and is therefore excluded. In the notation
of \eqref{eqn: boundary_at_fy}, we set

\[
\Omega_{\epsilon,l}(t)
:=\Omega_{f_{\epsilon,l}(\cdot,t)},
\qquad
\mathcal{B}_{\epsilon,l}(t)
:=\mathcal{B}_{f_{\epsilon,l}(\cdot,t)},
\qquad
\Omega_{\epsilon,l}^c(t)
:=\mathcal T\setminus\overline{\Omega_{\epsilon,l}(t)}.
\]

The associated nutrient and pressure fields are expanded as

\begin{subequations}
\label{eqn:asymt solns}
\begin{alignat}{1}
c_{\epsilon,l}^{\mathrm{(i)}}(x,y,t)
&=c_0^{\mathrm{(i)}}(x)
+\epsilon_l(t)\hat c_{1,l}^{\mathrm{(i)}}(x)\cos(ly)
+O(\epsilon_l^2),
\qquad \text{in }\Omega_{\epsilon,l}(t),\\
c_{\epsilon,l}^{\mathrm{(o)}}(x,y,t)
&=c_0^{\mathrm{(o)}}(x)
+\epsilon_l(t)\hat c_{1,l}^{\mathrm{(o)}}(x)\cos(ly)
+O(\epsilon_l^2),
\qquad \text{in }\Omega_{\epsilon,l}^c(t),\\
p_{\epsilon,l}(x,y,t)
&=p_0(x)
+\epsilon_l(t)\hat p_{1,l}(x)\cos(ly)
+O(\epsilon_l^2),
\qquad \text{in }\Omega_{\epsilon,l}(t).
\end{alignat}
\end{subequations}

Following the linearization leading to
\eqref{eqn:first_order_equation_radius}, we substitute
\eqref{eqn:asymt solns} into the differential equations
in \eqref{eqn:limitpressure} and \eqref{eqn:invivonutrient}.
The balance of the first-order terms yields
\begin{subequations}
\label{eqn: first_order_equation}
\begin{alignat}{1}
-\partial_x^2\hat c_{1,l}^{\mathrm{(i)}}
+(l^2+\lambda)\hat c_{1,l}^{\mathrm{(i)}}
&=0,\\
-\partial_x^2\hat c_{1,l}^{\mathrm{(o)}}
+(l^2+1)\hat c_{1,l}^{\mathrm{(o)}}
&=0,\\
-\partial_x^2\hat p_{1,l}
+(l^2+\gamma)\hat p_{1,l}
&=\frac{\hat c_{1,l}^{\mathrm{(i)}}}{c_B}.
\end{alignat}
\end{subequations}
The interior perturbations decay as \(x\to-\infty\), whereas the
exterior nutrient perturbation decays as \(x\to+\infty\).

The outward unit normal to
\(\mathcal{B}_{\epsilon,l}(t)\) satisfies
\begin{equation}
\label{eqn:normal_expansion_traveling}
n_{\epsilon,l}=e_x+O(\epsilon_l),
\end{equation}
where \(e_x\) denotes the unit vector in the \(x\)-direction.
\begin{samepage}
Using \eqref{eqn:asymt solns} and Taylor expansions
about \(x=0\), we approximate the field values and normal derivatives
on \(\mathcal{B}_{\epsilon,l}(t)\). Substitution into the nutrient
transmission conditions and the pressure boundary condition, together with the
zero-order conditions in \eqref{eqn: nutrients_traveling} and
\eqref{eqn: pressure_traveling}, yields the following boundary
conditions for the first-order terms:
\begin{subequations}
\label{eqn: first_order_boundary_condition}
\begin{alignat}{1}
\hat c_{1,l}^{\mathrm{(i)}}(0)
&=\hat c_{1,l}^{\mathrm{(o)}}(0),\\
\partial_x^2c_0^{\mathrm{(i)}}(0)
+\partial_x\hat c_{1,l}^{\mathrm{(i)}}(0)
&=\partial_x^2c_0^{\mathrm{(o)}}(0)
+\partial_x\hat c_{1,l}^{\mathrm{(o)}}(0),\\
\hat p_{1,l}(0)&=-\partial_xp_0(0)=\sigma_0.
\end{alignat}
\end{subequations}
\end{samepage}

Solving \eqref{eqn: first_order_equation} subject to
\eqref{eqn: first_order_boundary_condition}, we obtain the first-order
correction terms. For the nutrient,
\begin{subequations}
\label{eqn:first_order_nutrient_solution_traveling}
\begin{alignat}{1}
\hat c_{1,l}^{\mathrm{(i)}}(x)
&=-\frac{c_B\sqrt{\lambda}}
{\sqrt{\lambda+l^2}+\sqrt{1+l^2}}
e^{\sqrt{\lambda+l^2}x},
\qquad x\leq0,\\
\hat c_{1,l}^{\mathrm{(o)}}(x)
&=-\frac{c_B\sqrt{\lambda}}
{\sqrt{\lambda+l^2}+\sqrt{1+l^2}}
e^{-\sqrt{1+l^2}x},
\qquad x\geq0.
\end{alignat}
\end{subequations}

For \(x\leq0\) and \(\lambda\neq\gamma\), the bounded
pressure perturbation is
\begin{equation}
\label{eqn:first_order_pressure_solution_traveling}
\hat p_{1,l}(x)
=\frac{\sqrt{\lambda}\left(e^{\sqrt{\lambda+l^2}x}
-e^{\sqrt{\gamma+l^2}x}\right)}
{(\lambda-\gamma)(\sqrt{\lambda+l^2}+\sqrt{1+l^2})}
+\sigma_0e^{\sqrt{\gamma+l^2}x}.
\end{equation}
Setting \(\gamma=0\) in
\eqref{eqn:first_order_pressure_solution_traveling} recovers the result
of \cite{feng2023tumor}.

When \(\lambda=\gamma>0\), the apparent singularity in
\eqref{eqn:first_order_pressure_solution_traveling} is removable, and
for \(x\leq0\) the limiting expression is
\begin{equation}
\label{eqn:first_order_pressure_solution_resonant_traveling}
\hat p_{1,l}(x)
=\left[
\frac{1}{2\sqrt{\lambda}(\sqrt{\lambda}+1)}
+\frac{\sqrt{\lambda}\,x}
{2\sqrt{\lambda+l^2}
(\sqrt{\lambda+l^2}+\sqrt{1+l^2})}
\right]
e^{\sqrt{\lambda+l^2}x}.
\end{equation}

Together with \(p_0\), these first-order corrections
determine the boundary evolution function, whose derivation is
presented in the next subsection.

\subsubsection{Evolution equation and its properties}
\label{sec:perturbation_growth_rate_traveling}

As in the finite-radius case, the evolution equation for
\(\epsilon_l(t)\) is obtained by expanding Darcy's law on the perturbed
boundary to first order.
Since the unperturbed planar front propagates with speed
\(\sigma_0\), the normal velocity of the perturbed front is
\begin{equation}
\label{eqn:normal_velocity_traveling}
V_{\epsilon,l}
=\frac{\sigma_0+\dfrac{d\epsilon_l}{dt}\cos(ly)}
{\sqrt{1+l^2\epsilon_l^2\sin^2(ly)}}
=\sigma_0+\frac{d\epsilon_l}{dt}\cos(ly)
+O(\epsilon_l^2).
\end{equation}
On the other hand, the pressure expansion in
\eqref{eqn:asymt solns} and the normal expansion
\eqref{eqn:normal_expansion_traveling} give
\begin{equation}
\label{eqn:normal_pressure_expansion_traveling}
\nabla p_{\epsilon,l}\cdot n_{\epsilon,l}
=\partial_xp_0(0)
+\epsilon_l(t)
\left[
\partial_x^2p_0(0)+\partial_x\hat p_{1,l}(0)
\right]\cos(ly)
+O(\epsilon_l^2)
\end{equation}
on \(\mathcal{B}_{\epsilon,l}(t)\). Substituting
\eqref{eqn:normal_velocity_traveling} and
\eqref{eqn:normal_pressure_expansion_traveling} into the boundary
condition \(V_{\epsilon,l}=-\nabla p_{\epsilon,l}\cdot n_{\epsilon,l}\),
and using \(\sigma_0=-\partial_xp_0(0)\), yields
\begin{equation}
\label{eqn:amplitude_evolution_traveling}
\epsilon_l^{-1}\frac{d\epsilon_l}{dt}
=\underbrace{-\left[
\partial_x^2p_0(0)+\partial_x\hat p_{1,l}(0)
\right]}_{=:~\mathcal{E}_l(\lambda,\gamma)}
+O(\epsilon_l).
\end{equation}
Following the terminology used in the finite-radius
case, we call \(\mathcal{E}_l\) the boundary evolution function of the
\(l\)-th mode.
As before, \(\mathcal{E}_l>0\) indicates boundary
instability in the \(l\)-th mode, whereas \(\mathcal{E}_l<0\) indicates
stability with respect to that mode.

For \(\lambda\neq\gamma\), using the expressions for
\(p_0\) and \(\hat p_{1,l}\) in \eqref{eqn: pressure_solution_0} and
\eqref{eqn:first_order_pressure_solution_traveling}, respectively, a
direct calculation gives
\begin{equation}
\label{eqn: derivation_of_boundary}
\mathcal{E}_l(\lambda,\gamma)
=\frac{1}{\sqrt{\lambda}+1}
\left(1-\frac{\sqrt{\gamma+l^2}}
{\sqrt{\gamma}+\sqrt{\lambda}}\right)
-\frac{\sqrt{\lambda}}
{(\sqrt{\gamma+l^2}+\sqrt{\lambda+l^2})
(\sqrt{\lambda+l^2}+\sqrt{1+l^2})}.
\end{equation}
Formula \eqref{eqn: derivation_of_boundary} is regular at
\(\lambda=\gamma\) and agrees with the value obtained directly from
\eqref{eqn:first_order_pressure_solution_resonant_traveling}.

We next establish several mathematical properties of
the boundary evolution function \(\mathcal E_l\) and use them to characterize the
qualitative dependence of boundary instability on the mode number and
the model parameters. When considering the limit \(l\to0^+\), we
regard \(l\) as a positive real variable. Unlike in the
finite-radius case, \(l=1\) is a genuine shape perturbation of the
planar front; the admissible modes in the \(2\pi\)-periodic tube are the
integers \(l\geq1\).

\begin{prop}
\label{prop: derivation_boundary_evolution}
Let \(\lambda>0\) and \(\gamma>0\). The zero mode is neutral,
\begin{equation}
\label{eqn:zero_mode_neutral_traveling}
\mathcal{E}_0(\lambda,\gamma)=0.
\end{equation}
Moreover,
\begin{subequations}
\label{eqn:growth_rate_asymptotics_traveling}
\begin{align}
\mathcal{E}_l(\lambda,\gamma)
&=\frac{l^2}
{2(\sqrt{\lambda}+1)(\sqrt{\lambda}+\sqrt{\gamma})}
+O(l^4)
\qquad\text{as }l\to0^+,
\label{eqn:growth_rate_small_mode_traveling}\\
\mathcal{E}_l(\lambda,\gamma)
&=-\frac{l}
{(\sqrt{\lambda}+1)(\sqrt{\lambda}+\sqrt{\gamma})}
+\frac{1}{\sqrt{\lambda}+1}
+O(l^{-1})
\qquad\text{as }l\to+\infty.
\label{eqn:growth_rate_large_mode_traveling}
\end{align}
\end{subequations}
Consequently, sufficiently high-frequency modes are stable, and the
continuous modal growth-rate curve has a unique positive zero
\(l_*(\lambda,\gamma)>0\). More precisely,
\begin{equation}
\label{eqn:continuous_dispersion_sign_traveling}
\mathcal E_l(\lambda,\gamma)>0
\quad\text{for }0<l<l_*(\lambda,\gamma),
\qquad
\mathcal E_l(\lambda,\gamma)<0
\quad\text{for }l>l_*(\lambda,\gamma).
\end{equation}

For a fixed integer \(l\geq1\), let \(\lambda_0^l>l^2\) denote the
threshold for the model without pressure feedback, characterized by
\begin{equation}
\label{eqn:lambda_threshold_traveling}
\mathcal{E}_l(\lambda_0^l,0)=0,
\end{equation}
as established in \cite{feng2025nonsymmetric}. If
\(0<\lambda<\lambda_0^l\), then there exists a unique
\(\gamma_l^*(\lambda)>0\) such that
\begin{equation}
\label{eqn:gamma_threshold_traveling}
\mathcal{E}_l(\lambda,\gamma_l^*(\lambda))=0.
\end{equation}
Furthermore,
\begin{equation}
\label{eqn:gamma_threshold_sign_traveling}
\mathcal{E}_l(\lambda,\gamma)<0
\quad\text{for }0<\gamma<\gamma_l^*(\lambda),
\qquad
\mathcal{E}_l(\lambda,\gamma)>0
\quad\text{for }\gamma>\gamma_l^*(\lambda),
\end{equation}
and the crossing is transversal:
\begin{equation}
\label{eqn:gamma_transversality_traveling}
\partial_\gamma\mathcal{E}_l
(\lambda,\gamma_l^*(\lambda))>0.
\end{equation}
\end{prop}

\begin{proof}
The identity \eqref{eqn:zero_mode_neutral_traveling}
follows by setting \(l=0\) in \eqref{eqn: derivation_of_boundary}.
For \(l\to0^+\), Taylor expansion of each square root at \(l=0\)
directly gives \eqref{eqn:growth_rate_small_mode_traveling}. For
\(l\to+\infty\), we use
\begin{equation*}
\sqrt{l^2+\nu}=l+\frac{\nu}{2l}+O(l^{-3}),
\qquad \nu>0.
\end{equation*}
Consequently, the first term on the right-hand side of
\eqref{eqn: derivation_of_boundary} equals
\begin{equation*}
-\frac{l}{(\sqrt{\lambda}+1)(\sqrt{\lambda}+\sqrt{\gamma})}
+\frac{1}{\sqrt{\lambda}+1}+O(l^{-1}),
\end{equation*}
whereas the second term is \(O(l^{-2})\). This yields
\eqref{eqn:growth_rate_large_mode_traveling}.

We next prove the uniqueness of the positive zero of the
continuous modal growth-rate curve. The case \(\gamma=0\) was established in
\cite{feng2025nonsymmetric}, so it remains to consider \(\gamma>0\).
Set
\[
a=\sqrt{\lambda},\qquad b=\sqrt{\gamma},\qquad z=l^2,
\]
and define
\[
A(z)=\sqrt{a^2+z},\qquad
B(z)=\sqrt{b^2+z},\qquad
S(z)=\sqrt{1+z},\qquad
D(z)=a+b-B(z).
\]
For \(0\leq z<z_c:=a^2+2ab\), let
\[
Q(z):=D(z)(B(z)+A(z))(A(z)+S(z)).
\]
Multiplying \eqref{eqn: derivation_of_boundary} by its positive
denominator shows that
\begin{equation}
\label{eqn:continuous_growth_sign_transform_traveling}
\operatorname{sgn}\mathcal E_{\sqrt{z}}(\lambda,\gamma)
=\operatorname{sgn}\bigl(Q(z)-Q(0)\bigr),
\qquad 0\leq z<z_c,
\end{equation}
where \(Q(0)=a(a+b)(a+1)\). We claim that \(Q\) is strictly
log-concave on \([0,z_c)\). Direct differentiation gives
\begin{align}
\frac{d^2}{dz^2}\log Q(z)
={}&\frac{a+b-2B}{4B^3D^2}
-\frac{B^2+A^2}{4B^3A^3}
-\frac{A^2+S^2}{4A^3S^3},
\label{eqn:log_concavity_continuous_growth_traveling}
\end{align}
where the argument \(z\) has been suppressed on the right-hand side.
If \(a+b-2B\leq0\), the right-hand side is negative. If
\(a+b-2B>0\), then \(a>b\) and
\begin{equation}
\label{eqn:log_concavity_auxiliary_traveling}
A(a+b-2B)<D^2.
\end{equation}
To see this, set
\[
w=\frac{a+b-2B}{a+b}\in(0,1),
\qquad
\delta=\frac{a-b}{a+b}\in(0,1).
\]
Since \(A^2-B^2=(a-b)(a+b)\), inequality
\eqref{eqn:log_concavity_auxiliary_traveling} is equivalent to
\[
w\sqrt{\frac{(1-w)^2}{4}+\delta}
<\frac{(1+w)^2}{4}.
\]
This follows from \(\delta<1\) and
\[
(1+w)^4-4w^2\bigl((1-w)^2+4\bigr)
=(1-w)^2(-3w^2+6w+1)>0.
\]
Consequently, the first positive term in
\eqref{eqn:log_concavity_continuous_growth_traveling} is strictly
smaller than \(1/(4B^3A)\), which is in turn strictly smaller than the
absolute value of the second term. This proves the claim.

Moreover,
\[
\left.\frac{d}{dz}\log Q(z)\right|_{z=0}
=\frac{1}{2a}>0,
\qquad
\lim_{z\to z_c^-}Q(z)=0.
\]
Strict log-concavity therefore implies that \(Q\) first increases and
then decreases, and that the equation \(Q(z)=Q(0)\) has exactly one
solution in \((0,z_c)\). For \(z\geq z_c\), one has \(D(z)\leq0\),
and the numerator obtained in
\eqref{eqn:continuous_growth_sign_transform_traveling} is strictly
negative. This proves \eqref{eqn:continuous_dispersion_sign_traveling}.

For fixed \(l>0\), formula \eqref{eqn: derivation_of_boundary} gives
\begin{equation}
\label{eqn:large_gamma_asymptotic_traveling}
\mathcal{E}_l(\lambda,\gamma)
=\frac{\sqrt{\lambda}}{\sqrt{\lambda}+1}
\left(
1-\frac{\sqrt{\lambda}+1}
{\sqrt{\lambda+l^2}+\sqrt{1+l^2}}
\right)\frac{1}{\sqrt{\gamma}}
+O(\gamma^{-1})
\end{equation}
as \(\gamma\to+\infty\). The coefficient is positive for every
\(l>0\). Since
\(\mathcal{E}_l(\lambda,0)<0\) when
\(0<\lambda<\lambda_0^l\), continuity gives at least one positive zero.

It remains to prove uniqueness with respect to \(\gamma\).
We continue to write \(a=\sqrt{\lambda}\) and
\(b=\sqrt{\gamma}\), and introduce the additional auxiliary variables
\[
q=\sqrt{\lambda+l^2},
\quad k=\sqrt{\gamma+l^2},
\quad u=b+k,
\]
and define
\[
C:=\frac{a(a+1)}{q+\sqrt{1+l^2}},
\qquad
H(u):=
\frac{(au-l^2)(u^2+2qu+l^2)}
{u(u^2+2au-l^2)}.
\]
Then \(u\) is strictly increasing from \(l\) to \(+\infty\) as
\(\gamma\) increases from zero to infinity. A direct rearrangement of
\eqref{eqn: derivation_of_boundary} gives
\begin{equation}
\label{eqn:growth_rate_sign_transform_traveling}
\mathcal{E}_l(\lambda,\gamma)
=\frac{2u}
{(a+1)(u^2+2qu+l^2)}\bigl(H(u)-C\bigr).
\end{equation}
To check the monotonicity of \(H\), let \(e=q-a>0\), so that
\(l^2=e(e+2a)\). Differentiation gives
\[
H'(u)
=\frac{e^2P(u)}
{u^2(u^2+2au-l^2)^2},
\]
where
\begin{align*}
P(u)={}&u^4+4(e+2a)u^3
+(2e+4a)(2e+5a)u^2\\
&+4a(e+2a)^2u-e(e+2a)^3.
\end{align*}
The polynomial \(P\) is strictly increasing for \(u>0\), and
\[
P(l)
=l^4+4(q+a)l^3
+(q+a)(3q+5a)l^2
+4a(q+a)^2l>0.
\]
Hence \(H'(u)>0\) for \(u\geq l\). Moreover,
\(H(l)<C\) is equivalent to
\(\mathcal{E}_l(\lambda,0)<0\), whereas
\[
\lim_{u\to+\infty}H(u)=a>C,
\]
because \(q+\sqrt{1+l^2}>a+1\). Thus \(H(u)=C\) has exactly one
solution. Equation \eqref{eqn:growth_rate_sign_transform_traveling}
gives \eqref{eqn:gamma_threshold_sign_traveling}, and the strict
monotonicity of \(H\), together with the strict monotonicity of \(u\)
in \(\gamma\), gives the transversality condition
\eqref{eqn:gamma_transversality_traveling}.
\end{proof}

Figure \ref{fig:traveling_wave_growth_rate} illustrates these conclusions
for \(\lambda=0.5\). Its left-hand plot displays the unique-zero and sign
properties in \eqref{eqn:continuous_dispersion_sign_traveling}: at
\(\gamma=10\), the continuous modal growth-rate curve crosses the horizontal
axis only once, at \(l\approx1.86\). Among the admissible integer modes
shown, only \(l=1\) is unstable. For \(l=2\), the pressure-feedback
threshold is \(\gamma_2^*(0.5)\approx13.12\). The growth rate is not
globally increasing in \(\gamma\): after crossing zero, it reaches a
positive maximum near \(\gamma=60.68\) and then decreases toward zero
from above, consistently with
\eqref{eqn:large_gamma_asymptotic_traveling}.

\begin{center}
\includegraphics[width=\textwidth]{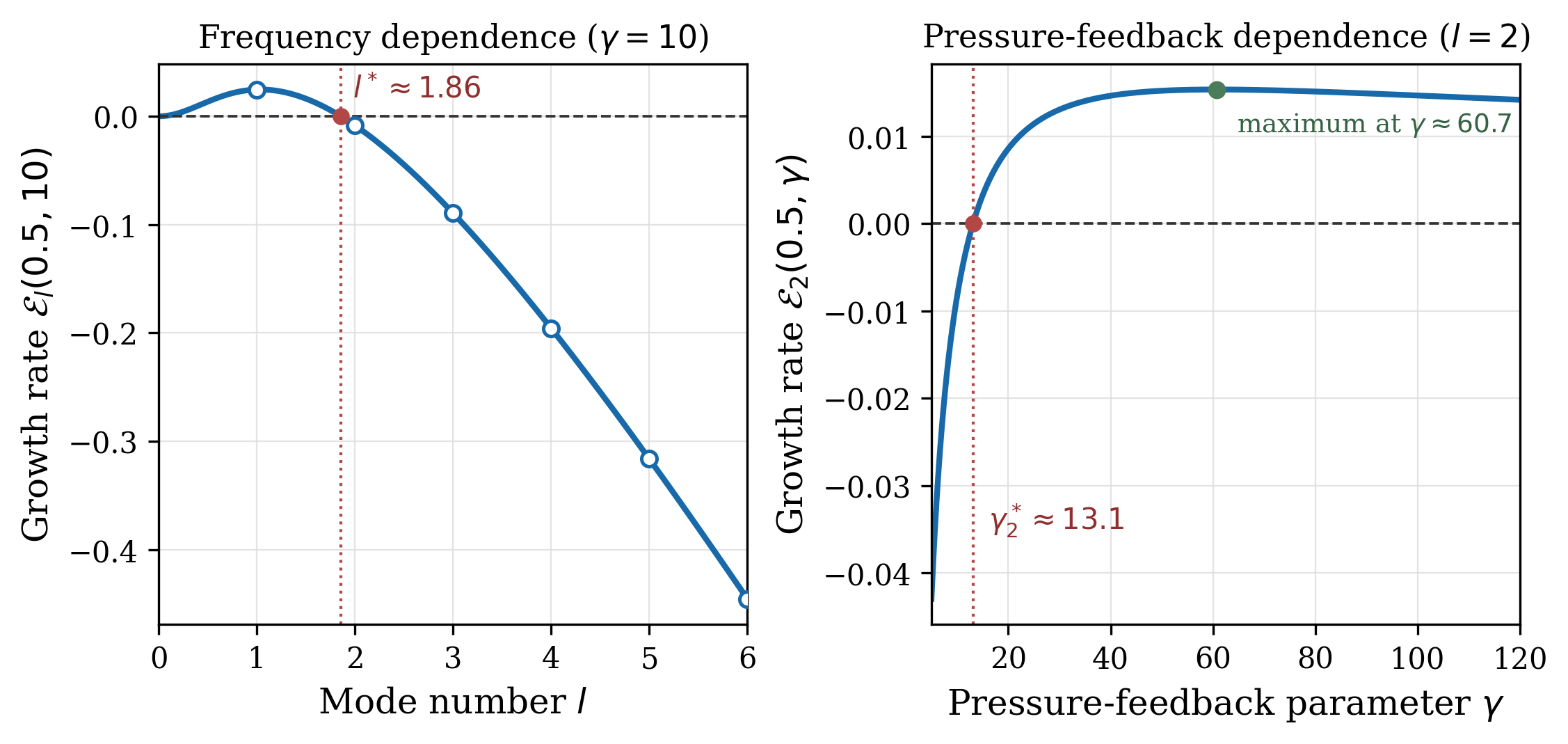}
\captionof{figure}{The traveling-wave growth rate for
\(\lambda=0.5\). Left: \(\mathcal{E}_l(0.5,10)\) for
\(0\leq l\leq6\), illustrating its unique positive zero; the open
circles mark the admissible integer modes.
Right: \(\mathcal{E}_2(0.5,\gamma)\) for
\(5\leq\gamma\leq120\), showing the unique stability threshold and
the subsequent nonmonotone dependence on \(\gamma\).}
\label{fig:traveling_wave_growth_rate}
\end{center}

Proposition
\ref{prop: derivation_boundary_evolution} makes the destabilizing effect
of pressure feedback more precise. When \(\gamma=0\) and
\(0<\lambda<\lambda_0^l\), the \(l\)-th mode is stable (see the
left-hand plot of Figure 1 in \cite{feng2025nonsymmetric}). After
pressure feedback is introduced, the same mode becomes unstable once
\(\gamma>\gamma_l^*(\lambda)\), even though \(\lambda\) remains below
the threshold in the absence of pressure feedback \(\lambda_0^l\).

\begin{rem}\leavevmode
\begin{enumerate}[(1)]
\item
The boundary-instability results obtained here share a qualitative conclusion
with Ye and Lin \cite{ye2024fingering}: pressure-dependent inhibition of cell
growth does not by itself guarantee stability of an advancing interface. The
underlying mechanisms are nevertheless different.
Their model couples the cell phase to an exterior passive fluid, and both the
friction contrast and the Laplace--Young condition enter the modal growth-rate
formula; in particular, surface tension stabilizes sufficiently
high-frequency boundary perturbations. In the present model, \(p=0\) on the
tumor boundary and no curvature term is imposed, so high-frequency stability
instead follows from the interior nutrient--pressure response and the induced
Darcy velocity.

\item
At \(\gamma=\gamma_l^*(\lambda)\), the \(l\)-th mode is
neutral at linear order: the boundary perturbation \(\cos(ly)\) neither
grows nor decays. This suggests that suitable higher-order corrections
to this boundary profile and to the corresponding nutrient and pressure
fields may produce a steadily propagating nonplanar solution of the full
nonlinear problem. In the next section, we verify this intuition and
establish the existence of nonsymmetric traveling waves near
\(\gamma_l^*(\lambda)\).
\end{enumerate}
\end{rem}

\section{Existence of nonsymmetric traveling wave solutions}
\label{sec:traveling_wave_bifurcation}

{
{Motivated by the preceding linear analysis, in this
section we prove the existence of nonsymmetric traveling waves by a
local bifurcation argument. More precisely, for each integer \(l\geq1\)
and \(0<\lambda<\lambda_0^l\), we show that the critical value
\(\gamma_l^*(\lambda)\) identified in Proposition
\ref{prop: derivation_boundary_evolution} is a bifurcation point on the
branch of planar traveling waves and gives rise to a local branch of
nonsymmetric traveling-wave solutions.}

{To establish this bifurcation result, we first need to
justify the asymptotic expansions in
\eqref{eqn:asymt solns} in appropriate function spaces. Although the
required arguments are standard, they involve
lengthy technical estimates. To keep the main argument clear, we
present the essential steps and omit some technical details. The
linearization and the conditions required by the bifurcation theorem
are verified explicitly.}
}

\subsection{A nonlinear functional map}
\label{sec:nonlinear_functional_traveling}

{
{Let \(f\) be a \(2\pi\)-periodic even boundary profile.
As in \eqref{eqn: boundary_at_fy}, let \(\Omega_f\) and
\(\mathcal B_f\) denote the corresponding tumor region and boundary,
and define the exterior region by
\(\Omega_f^c:=\mathcal T\setminus\overline{\Omega_f}\). Unlike in
\cite{feng2025nonsymmetric}, we regard \(\gamma\) as the bifurcation
parameter and fix \(\lambda>0\) throughout this section.}

{For each \(\gamma>0\), we freeze the interface at
\(\mathcal B_f\) and denote by \(c_f^{\mathrm{(i)}}\),
\(c_f^{\mathrm{(o)}}\), and \(p_f^{\mathrm{(i)}}\) the corresponding
nutrient and interior pressure fields, with their dependence on
\(\lambda\) and \(\gamma\) suppressed in the notation. They solve the
nutrient transmission problem \eqref{eqn:invivonutrient} in
\(\Omega_f\) and \(\Omega_f^c\), and the elliptic pressure problem
\eqref{eqn:pressure_equation}--\eqref{bc:p_zero} in \(\Omega_f\),
together with the far-field conditions used in Section
\ref{sec:traveling_wave_scenario}. For a general profile \(f\), these
fields should be understood as the instantaneous solutions associated
with the prescribed interface. The resulting Darcy velocity need not
agree with a single propagation speed along \(\mathcal B_f\); hence
\(f\) need not propagate with a fixed shape.}

{For a general profile \(f\) with \(O(1)\) amplitude,
finding a traveling wave that preserves its boundary shape is difficult.
We therefore seek profiles with \(O(\epsilon)\) amplitude near the
planar front that propagate without changing shape. In the phase space
of boundary profiles and the bifurcation parameter \(\gamma\), these
small-amplitude shape-preserving waves correspond to a local
bifurcation near the branch of planar
traveling-wave solutions. To prove the existence of such a local
bifurcation, motivated by the calculation in
\eqref{eqn:normal_velocity_traveling} and related bifurcation studies
\cite{borisovich2005symmetry,feng2025nonsymmetric}, we consider the
following nonlinear functional map:}
}
\begin{equation}
\label{eqn: nonlinear_functional_map}
F(f,\gamma)(y)
:=-\partial_xp_f^{\mathrm{(i)}}(f(y),y)
-\sigma_0(\gamma).
\end{equation}
{
{Here \(\sigma_0(\gamma)\) denotes the propagation speed
of the planar traveling wave given in
\eqref{eqn: symmetric_traveling_speed}, viewed as a function of
\(\gamma\) with \(\lambda\) fixed. Then the planar fronts form the
trivial branch}
}
\begin{equation}
\label{eqn:trivial_branch_traveling}
F(0,\gamma)=0,
\qquad \gamma>0,
\end{equation}
{
{whereas, after excluding the translational mode as in
the next subsection, nonsymmetric traveling waves correspond to the
nontrivial solutions}
}
\begin{equation}
\label{eqn: functional_equation}
F(f,\gamma)=0,
\qquad f(y)\not\equiv0.
\end{equation}

{
{The Crandall--Rabinowitz theorem (see Theorem
\ref{thm: crandall-robinowitz} below) provides a natural way to obtain
such a local branch. The first step is to determine the
Fr\'{e}chet derivative of \(F\) with respect to the boundary profile,
which is the subject of the next subsection.}}

\subsection{The Fr\'{e}chet derivative of the functional map}
\label{sec:frechet_derivative_traveling}

For \(0<\alpha<1\) and an integer \(m\geq0\), define the classical
H\"older space
\begin{equation}
\label{eqn:function_spaces_traveling}
X^{m+\alpha}(\mathbb R)
:=\left\{f(y)\in C^{m+\alpha}(\mathbb R):
f(y)\text{ is }2\pi\text{-periodic and even}\right\}.
\end{equation}
{
We denote by \(X_1^{m+\alpha}(\mathbb R)\) the closure in
\(X^{m+\alpha}(\mathbb R)\) of the linear space spanned by the cosine
modes \(\cos(jy)\), \(j\geq1\). Thus, the constant mode, which
represents a uniform translation of the planar front, is excluded from
the perturbation space.  {Since we seek a local
bifurcation from the planar solution \(f(y)\equiv0\), we write}
\begin{equation*}
{
f(y)=\epsilon\widetilde f(y),
\qquad |\epsilon|\ll1,}
\end{equation*}
{where \(\widetilde f\in
X_1^{3+\alpha}(\mathbb R)\) is fixed independently of \(\epsilon\).
Without loss of generality, we normalize it by}
\[
\|\widetilde f\|_{X_1^{3+\alpha}(\mathbb R)}=1.
\]
{Indeed, any fixed nonzero amplitude can be absorbed
into \(\epsilon\).}
}
Locally near the planar branch, {after projection onto
the nonconstant cosine modes,} the map in
\eqref{eqn: nonlinear_functional_map} is regarded as
\[
F:X_1^{3+\alpha}(\mathbb R)\times(0,\infty)
\longrightarrow X_1^{2+\alpha}(\mathbb R).
\]

{
{As in Lemma 3.2 and Lemma 3.3 of
\cite{feng2025nonsymmetric}, the first-order expansions of \(c\) and
\(p\) in \eqref{eqn:asymt solns}, used in Section
\ref{sec:single_mode_perturbation_traveling} and extended here to
general profiles by Fourier superposition, can be justified
rigorously in the H\"older spaces defined above by straightening the
perturbed boundary and applying standard Schauder estimates to the
resulting fixed-domain elliptic problems. Using these rigorous
expansions, one can verify that there exists a constant \(C>0\),
locally uniform in \(\gamma\) and independent of \(\epsilon\) and of
the choice of normalized profile \(\widetilde f\), such that}
}
\begin{equation}
\label{eqn:functional_frechet_remainder_traveling}
\Big\|F(\epsilon\widetilde f,\gamma)-F(0,\gamma)
-\epsilon\bigl(
-\widetilde f\,\partial_x^2p_0(0)
-\partial_xp_1(0,\cdot;\widetilde f)
\bigr)\Big\|_{X_1^{2+\alpha}(\mathbb R)}
\leq C\epsilon^2.
\end{equation}
{
{Here \(p_0\) is the planar pressure defined in
\eqref{eqn: pressure_solution_0}, and
\(p_1(x,y;\widetilde f)\) denotes the first-order interior pressure
correction generated by the profile \(\widetilde f\). In particular,
if \(\widetilde f(y)=\cos(ly)\), then
\(p_1(x,y;\widetilde f)=\hat p_{1,l}(x)\cos(ly)\).}

{Since the expression in parentheses is bounded and
linear in \(\widetilde f\), and the constant in
\eqref{eqn:functional_frechet_remainder_traveling} is uniform over
normalized profiles, estimate
\eqref{eqn:functional_frechet_remainder_traveling} gives the
Fr\'{e}chet derivative of \(F\) with respect to \(f\) at
\((0,\gamma)\).}  We characterize it in the following proposition.
}

\begin{prop}
\label{prop: frechet_derivative}
For fixed \(\lambda>0\) and \(\gamma>0\), the Fr\'{e}chet derivative
of \(F\) with respect to \(f\) at \((0,\gamma)\) is the bounded
linear operator, {denoted by \(F_f(0,\gamma)\),} from
\(X_1^{3+\alpha}(\mathbb R)\) to
\(X_1^{2+\alpha}(\mathbb R)\) given by
\begin{equation}
\label{eqn: Frechet_derivative}
\left[F_f(0,\gamma)\right]\widetilde f
=-\widetilde f\,\partial_x^2p_0(0)
-\partial_xp_1(0,\cdot;\widetilde f).
\end{equation}
{
For \(\widetilde f\in X_1^{3+\alpha}(\mathbb R)\), write
}
\begin{equation}
\label{eqn:fourier_profile_frechet_traveling}
\widetilde f(y)
=\sum_{l=1}^{\infty}a_l\cos(ly),
\end{equation}
then
\begin{equation}
\label{eqn: eigenvalue_problem}
\left[F_f(0,\gamma)\right]\widetilde f
=\sum_{l=1}^{\infty}
a_l\mathcal E_l(\lambda,\gamma)\cos(ly),
\end{equation}
where \(\mathcal E_l(\lambda,\gamma)\) is given by
\eqref{eqn: derivation_of_boundary}.
\end{prop}

{
{This Fourier representation connects the bifurcation problem
to the properties of \(\mathcal E_l\) established in Proposition
\ref{prop: derivation_boundary_evolution}.}
}

\subsection{Application of the Crandall--Rabinowitz theorem}
\label{sec:crandall_rabinowitz_traveling}

To establish the existence of a local bifurcation of
the nonlinear map \eqref{eqn: nonlinear_functional_map} from the
trivial branch \eqref{eqn:trivial_branch_traveling}, we verify the
hypotheses of the celebrated Crandall--Rabinowitz theorem using the
Fr\'{e}chet derivative obtained in the previous subsection
\cite{crandall1971bifurcation}. For the reader's convenience, we recall
the theorem below.

\begin{thm}[Crandall--Rabinowitz]
\label{thm: crandall-robinowitz}
Let \(W\) and \(Z\) be real Banach spaces, and let
\(\mathcal F:W\times\mathbb R\to Z\) be a \(C^k\) mapping,
\(k\geq3\), in a neighborhood of \((0,\mu_0)\). Assume that
\begin{enumerate}
\item \(\mathcal F(0,\mu)=0\) for \(\mu\) near \(\mu_0\);
\item
\(\ker \mathcal F_w(0,\mu_0)=\operatorname{span}\{w_0\}\);
\item \(\operatorname{Range}\mathcal F_w(0,\mu_0)\) is closed and has
codimension one in \(Z\);
\item
\(\mathcal F_{w\mu}(0,\mu_0)w_0\notin
\operatorname{Range}\mathcal F_w(0,\mu_0)\).
\end{enumerate}
Then the nontrivial solutions of \(\mathcal F(w,\mu)=0\) near
\((0,\mu_0)\), together with the bifurcation point, form a
\(C^{k-2}\) curve \((w(\epsilon),\mu(\epsilon))\) satisfying
\[
w(\epsilon)=\epsilon w_0+o(\epsilon),
\qquad
\mu(0)=\mu_0.
\]
Locally, every solution lies either on this curve or on the trivial
branch \((0,\mu)\).
\end{thm}

We now apply this theorem to the nonlinear map \(F\)
defined in \eqref{eqn: nonlinear_functional_map}. The Fourier
representation in Proposition \ref{prop: frechet_derivative} and the
properties of \(\mathcal E_l\) established in Proposition
\ref{prop: derivation_boundary_evolution} allow us to verify its
hypotheses and thereby obtain a sequence of local bifurcation branches
indexed by the integers \(l\) satisfying
\(0<\lambda<\lambda_0^l\), where the threshold in the absence of pressure feedback
\(\lambda_0^l\) is defined by
\eqref{eqn:lambda_threshold_traveling}. As explained above, each branch
yields a family of nonsymmetric traveling-wave solutions of
\eqref{eqn:limitpressure}--\eqref{eqn:invivonutrient}. We summarize and
prove this conclusion in the following theorem.

\begin{thm}
\label{thm:traveling_wave_bifurcation}
For each positive integer \(l\) and
\(0<\lambda<\lambda_0^l\), let \(\gamma_l^*(\lambda)\) denote the
unique positive value defined by
\eqref{eqn:gamma_threshold_traveling}. Then
\((0,\gamma_l^*(\lambda))\) is a bifurcation point of
\eqref{eqn: functional_equation}. More precisely, there exist
\(\epsilon_0>0\) and a local curve of solutions
\[
(-\epsilon_0,\epsilon_0)\ni\epsilon
\longmapsto(f_l(\epsilon),\gamma(\epsilon))
\]
such that
\begin{equation}
\label{eqn:bifurcating_branch_traveling}
F(f_l(\epsilon),\gamma(\epsilon))=0,
\qquad
f_l(\epsilon)(y)=\epsilon\cos(ly)+o(\epsilon)
\quad\text{in }X_1^{3+\alpha}(\mathbb R),
\end{equation}
where \(f_l(0)(y)\equiv0\) and
\(\gamma(0)=\gamma_l^*(\lambda)\).
For \(0<|\epsilon|<\epsilon_0\), this curve gives nonsymmetric traveling-wave
profiles bifurcating from the planar branch.
\end{thm}
\begin{proof}
We apply Theorem \ref{thm: crandall-robinowitz} to the
nonlinear map \eqref{eqn: nonlinear_functional_map} with
\begin{align*}
W&=X_1^{3+\alpha}(\mathbb R),
&Z&=X_1^{2+\alpha}(\mathbb R),
&\mu_0&=\gamma_l^*(\lambda),\\
\mathcal F&=F,
&\mu&=\gamma,
&w_0&=\cos(ly).
\end{align*}
The required \(C^k\)-smoothness of \(F\) follows from
the same Schauder-estimate argument used above to justify the
asymptotic expansions. Moreover, \eqref{eqn:trivial_branch_traveling} gives
\(F(0,\gamma)=0\) for all \(\gamma>0\), so the first hypothesis of
Theorem \ref{thm: crandall-robinowitz} holds.

To verify the second hypothesis, let \(L\) denote the
Fr\'{e}chet derivative \(F_f(0,\gamma_l^*(\lambda))\). For
\(\widetilde f=\sum_{j\geq1}a_j\cos(jy)\), Proposition
\ref{prop: frechet_derivative} gives
\[
L\widetilde f
=\sum_{j\geq1}a_j
\mathcal E_j(\lambda,\gamma_l^*(\lambda))\cos(jy).
\]
By \eqref{eqn:gamma_threshold_traveling},
\(\mathcal E_l(\lambda,\gamma_l^*(\lambda))=0\). Moreover,
Proposition \ref{prop: derivation_boundary_evolution} shows that the
function
\(j\mapsto\mathcal E_j(\lambda,\gamma_l^*(\lambda))\), viewed for
real \(j>0\), has a unique positive zero. Hence
\(\mathcal E_j(\lambda,\gamma_l^*(\lambda))\neq0\) for every positive
integer \(j\neq l\). Consequently,
\[
\ker L=\operatorname{span}\{\cos(ly)\}.
\]
Therefore, the second hypothesis holds.

On the other hand, since the \(l\)-th eigenvalue is the only one that vanishes, the same
diagonal representation gives
\[
\operatorname{Range}L
=\left\{g\in X_1^{2+\alpha}(\mathbb R):
\text{the \(l\)-th cosine coefficient of \(g\) is zero}\right\}.
\]
Indeed, \eqref{eqn:growth_rate_large_mode_traveling} shows that the
inverse multipliers for \(j\neq l\) are of order \(j^{-1}\), so the
above characterization follows from standard Fourier multiplier
estimates.  The range is therefore closed and has codimension one.
Thus, the third hypothesis holds.

Finally, differentiating the eigenvalue relation in \(\gamma\) gives
\[
F_{f\gamma}(0,\gamma_l^*(\lambda))\cos(ly)
=\partial_\gamma\mathcal E_l
(\lambda,\gamma_l^*(\lambda))\cos(ly).
\]
The coefficient on the right is strictly positive by
\eqref{eqn:gamma_transversality_traveling}.  The resulting function
has a nonzero \(l\)-th cosine coefficient and therefore does not belong
to \(\operatorname{Range}L\).  This verifies the last
condition. At this point, all four hypotheses of Theorem
\ref{thm: crandall-robinowitz} have been verified, and the stated branch
follows.
\end{proof}

\begin{rem}
{
Zhao and Shi \cite{zhao2025determination} showed, for a related
free-boundary tumor model, that carrying the asymptotic expansion to
higher orders identifies the nontrivial branch as a pitchfork
bifurcation.  To obtain the corresponding conclusion for the present
model, one would need to compute higher-order corrections to both the
nutrient and pressure fields and determine the first nonzero
coefficient in the expansion of
\(\gamma(\epsilon)-\gamma_l^*(\lambda)\).  We do not pursue this
higher-order calculation here.
}
\end{rem}

\section{Conclusion and outlook}

In this paper, we have investigated the influence of
pressure feedback on boundary instability in a nutrient-coupled
Hele--Shaw tumor growth model obtained as the incompressible limit of a
porous-medium model. The incompressible limit preserves the
patch structure of the tumor density, so the evolution is described by
a single free boundary. Using asymptotic analysis, we derived the
boundary evolution functions for finite-radius tumors and planar
traveling fronts. For finite-radius tumors, we established the
noncommuting limits \eqref{eqn:noncommuting_limits_radius} of the
evolution function, which indicate that pressure feedback
qualitatively changes the large-radius stability: for
\(0<\lambda\leq1\) and \(\gamma>0\), all single-Fourier-mode
perturbations become unstable as the tumor radius becomes sufficiently
large, whereas the model without pressure feedback remains stable. For
planar traveling
fronts, when \(0<\lambda<\lambda_0^l\), the \(l\)-th mode is stable at
\(\gamma=0\) but loses stability at the unique positive threshold
\(\gamma_l^*(\lambda)\). We further proved that each threshold gives
rise to a local branch of nonsymmetric traveling waves. All these
results indicate that
pressure feedback can induce boundary instabilities that are absent in
the model without pressure feedback.

Compared with the models without pressure feedback studied in
\cite{feng2023tumor,feng2025nonsymmetric}, the present model includes an
inhibitory pressure term associated with the homeostatic pressure.
Although this term suppresses cell proliferation, it does not destroy
the patch structure and therefore does not produce a necrotic core.
More general incompressible tumor models can describe the formation of
a necrotic core through an obstacle problem for the pressure
\cite{dou2024tumor,feng2026tumor}. In these models, the outer tumor
boundary remains a Darcy-law-driven Hele--Shaw free boundary, whereas
the inner necrotic interface is determined implicitly at each time by
the free boundary of the pressure obstacle problem. A natural direction
for future work is to investigate the stability of such two-interface
configurations and the interaction between perturbations of the outer
tumor boundary and the inner necrotic interface.

\bibliographystyle{plain}
\bibliography{ref}

\end{document}